\RequirePackage[l2tabu, orthodox]{nag} 
\documentclass[11pt, a4paper]{article}

\title{GMRES Methods for Tomographic Reconstruction\\with an Unmatched Back Projector\thanks{We acknowledge financial support 		from Japan Society for the Promotion of Science grant no.~S19008 and JP20K14356, and from The Villum Foundation through Villum Investigator grant no.~25893.}}
\author{Per Christian Hansen\thanks{Department of Applied Mathematics and Computer Science, Technical University of Denmark, \texttt{pcha@dtu.dk}} \and Ken Hayami\thanks{Professor Emeritus, National Institute of Informatics, \texttt{hayami@nii.ac.jp}} \and Keiichi Morikuni\thanks{Faculty of Engineering, Information and Systems, University of Tsukuba, Japan, \texttt{morikuni@cs.tsukuba.ac.jp}}}
\date{}

\usepackage{amsmath, amsfonts, amssymb, amsthm, multirow}
\usepackage[all, warning]{onlyamsmath}
\usepackage[bookmarks=true, bookmarksnumbered=true, bookmarkstype=toc, colorlinks={true}, 	pdfauthor={Keiichi Morikuni}, pdfdisplaydoctitle={true}, pdfkeywords={Add keywords}, pdfpagemode=UseNone, pdfstartview=FitH, pdfsubject={subtitle}, pdftitle={Report}, setpagesize={false},	urlcolor={red}]{hyperref}
\usepackage[T1]{fontenc} 
\usepackage[top=1in, bottom=1in, left=1in, right=1in]{geometry}
\usepackage{exscale, fix-cm, lmodern, textcomp}
\usepackage{algorithm, algorithmic}
\usepackage{empheq}
\usepackage{booktabs}

\usepackage[caption=false]{subfig}

\usepackage{chngcntr}
\usepackage{apptools}
\AtAppendix{\counterwithin{lemma}{section}}
\newtheorem{lemma}{Lemma}

\usepackage{mathtools}
\mathtoolsset{showonlyrefs=true}

\hypersetup{pdfpagemode=UseNone} 

\theoremstyle{plain}
\newtheorem{theorem}{Theorem}[section]
\newtheorem{corollary}[theorem]{Corollary}

\numberwithin{equation}{section}
\makeatletter
	
	\@addtoreset{equation}{section}
\makeatother

\makeatletter
  
  \@addtoreset{algorithm}{section}
\makeatother

\makeatletter
  
  \@addtoreset{figure}{section}
\makeatother

\makeatletter
  
  \@addtoreset{table}{section}
\makeatother

\renewcommand{\arraystretch}{0.94}
\newcommand{\trans}{{\mathsf{T}}}
\newcommand{\be}{\varepsilon}
\newcommand{\nul}{{\cal N}}

\def\AB{A\hspace{1.9pt}B}
\def\BA{B\hspace{1.9pt}A}
\def\Bb{B\hspace{1.9pt}b}
\def\range{\mathcal{R}}
\def\As{A_{\mathtt{s}}^{}}
\def\Al{A_{\mathtt{l}}^{}}
\def\Ai{A_{\mathtt{i}}^{}}
\def\AsT{A_{\mathtt{s}}^\trans}
\def\AlT{A_{\mathtt{l}}^\trans}
\def\AiT{A_{\mathtt{i}}^\trans}
\def\fro{_\mathsf{F}}

\allowdisplaybreaks

\begin{document}
\maketitle

\begin{center}
\emph{P. C. Hansen and K. Morikuni dedicate this paper to K. Hayami.}
\end{center}

\begin{abstract}
Unmatched pairs of forward and back projectors are common in X-ray CT computations
for large-scale problems; they are caused by
the need for fast algorithms that best utilize the computer hardware, and it
is an interesting and challenging task to develop fast and easy-to-use
algorithms for these cases.
Our approach is to use preconditioned GMRES, in the form of the
AB- and BA-GMRES algorithms, to handle the unmatched normal equations
associated with an unmatched pair.
These algorithms are simple to implement, they rely only on computations
with the available forward and back projectors, and they do not require the
tuning of any algorithm parameters.
We show that these algorithms are equivalent to well-known LSQR and
LSMR algorithms in the case of a matched projector.
Our numerical experiments demonstrate that AB- and BA-GMRES exhibit
a desired semi-convergence behavior that is comparable with
LSQR/LSMR and that standard stopping rules work well.
Hence, AB- and BA-GMRES are suited for
large-scale CT reconstruction problems with noisy data and
unmatched projector pairs.
\end{abstract}

convergence, unmatched backprojector, unmatched normal equations, stopping rules

\section{Introduction}
\label{sec:Intro}

Computational algorithms for X-ray computed tomography (CT) come in different forms.
Some methods take their basis in an explicit formulation of the inverse operator,
leading to the filtered back projection algorithm (for parallel-beam CT),
the FDK algorithm (for cone-beam CT), etc.
Other methods are based on a discretization
of the problem followed by solving the -- usually over- or underdetermined --
linear system of equations by means of an iterative method.
The latter approach, which is the basis for this work, is more general in the sense
that it does not assume any specific scanning geometry, and it tends to produce
better reconstructions in the case of limited-data and/or limited-angle problems
\cite{SIAMbook}.

CT reconstruction is an inverse problem where the forward problem refers to the
mapping of the object's linear attenuation coefficient to the data in the form of
projections of the object onto the detector planes for various scan positions.
In the case of 2D parallel-beam CT the forward operator is known as the Radon transform.
Data consists of (noisy) measurements of the attenuation of the X-rays
through the object, recorded in a set of detector elements.
Discretization of this problem takes the form
\begin{equation}
	\label{eq:Axb}
	A\, x \approx b \ , \qquad b = \bar{b} + e \ , \qquad \bar{b} = A\,\bar{x} \ ,
	\qquad A \in \mathbb{R}^{m\times n} \ ,
\end{equation}
where $A$ is a discretization of the \emph{forward projector},
$b$ is the measured data, and $x$ represents the reconstructed image of
the object's interior.
Moreover, $\bar{x}$ represents the exact object, $\bar{b}$ represents
the noise-free data, and $e$
represents the measurement noise.
A number of discretization schemes are available for computing the matrix $A$,
see, e.g.,
\cite[Chapter~9]{SIAMbook}, \cite{HSSHN} and \cite[Appendix~A]{LaWe04}.
A discussion of the noise in CT can be found in \cite[Section~4.4]{SIAMbook};
here we assume that $e$ is white Gaussian noise.

The matrix $A$ is sparse and there are no restrictions on its dimensions $m$ and $n$;
both overdetermined and underdetermined systems are used, depending on the
measurement setup.
In both cases it is common to consider a least squares solution
which involves the normal equations in one of the forms
\[
A^\trans A\, x = A^\trans b \qquad \Longleftrightarrow \qquad
\min_x \| b - A\, x \|_2
\]
and
\[
A\, A^\trans y = b \ , \quad  x = A^\trans y  \qquad \Longleftrightarrow \qquad
\min_x \| x \|_2 \quad \hbox{subject to} \quad A\, x=b \ .
\]
The matrix $A^\trans$ represents the so-called \emph{back projector} which
maps the data back onto the solution domain~\cite{Natterer}.
The back projector is a mathematical abstraction that has no physical interpretation,
but it plays a central role in filtered back projection and many other
reconstruction algorithms.

In large-scale CT  problems, the matrix $A$ -- in spite of
the fact that it is sparse -- may be too large to store explicitly.
Instead, one must use subroutines or functions that compute the multiplications with $A$
and $A^\trans$ in a matrix-free fashion, often using graphics processing units (GPUs)
or other hardware accelerators.
In order to make best use of this hardware and the memory hierarchy
it is common to use different discretization techniques for the
forward projector and the back projector \cite{PBS11}.
This means that the matrix $B \in \mathbb{R}^{n\times m}$
which represents the back projector is
typically different from the transpose $A^\trans$ of the forward projector,
and we say that $B$ is an \emph{unmatched back projector}
or an \emph{unmatched transpose}.

A consequence of this is that iterative solvers, which are supposed
to solve the normal equations, instead solve
the so-called \emph{unmatched normal equations} in one of the forms
\begin{equation}
	\label{eq:unmatchednormalequations}
	\BA\, x = \Bb \qquad \hbox{and} \qquad A\, B \, y = b \ , \quad
	x = B \, y \ , \qquad B \in \mathbb{R}^{n\times m} \ ,
\end{equation}
see \cite{ElHa18} and \cite{DHHR19} for details.
The main drawback of using an unmatched transpose $B$ is that the standard
Simultaneous Iterative Reconstruction Technique (SIRT) iterative solvers
(Landweber, Cimmino, CAV, DROP, SART)~\cite{AIRtoolsII} do not converge
when the iteration matrix $\BA$
has one or more complex eigenvalues with a negative real part -- which
is very often the case in practice
(see Figure~\ref{fig:eigenvalues} in Section~\ref{sec:eigenvalues}).
A convergence analysis of proximal gradient methods with an unmatched
transpose is given in~\cite{CPRST21}.

As shown in \cite{DHHR19}, one can modify the SIRT methods such that
they solve a slightly modified problem whose iteration matrix has no
eigenvalues with a negative real part, thus ensuring convergence.
But this does introduce a (small) perturbation of the solution, and one
must compute an estimate of the leftmost eigenvalue(s) of $\BA$.
In addition, the choice of a good relaxation parameter is nontrivial.

An alternative is to use Kaczmarz's method \cite{AIRtoolsII} which does not
explicitly involve the matrix transpose; but this method has other drawbacks.
In its original form, Kaczmarz's method operates on a single row of $A$ at
a time, which is unsuited for modern computer architectures.
Block versions \cite{SoHa14} have better
performance but they require a good choice of the blocks, the blocks
may be unmatched, and again the choice of a good relaxation parameter is nontrivial.

The present work explores an alternative approach, where we use the well-known
GMRES algorithm to solve the unmatched normal
equations \eqref{eq:unmatchednormalequations} with an unmatched
back projector.
Thus, we avoid introducing a perturbation as well as the need for eigenvalue
computations for the sake of ensuring convergence of the SIRT methods,
and we avoid the choice of the relaxation parameter.
This makes it easier to develop a black-box CT solver that does not rely
on the user's knowledge of convergence criteria, eigenvalue computations, etc.
Our work is based on previous work on the preconditioned
AB- and BA-GMRES methods for solving least squares problems \cite{HaYI10}
and here we explicitly demonstrate the successful use of these methods for the
CT reconstruction problems.

An unmatched transpose also arises in connection with image
deblurring with anti-reflective boundary conditions.
There, discretizations of the blurring operator and its adjoint
that incorporate the desired boundary conditions lead to
an unmatched pair of matrices, see \S\S 3.3--3.4 in \cite{DoEsMaSe06}.
The corresponding unmatched normal equations are solved by means
of GMRES in \cite{DoMaRe15}, similar to the present work.
In the present paper, the unmatched transpose comes from an adjoint
(the back projector) with a different discretization from the forward operator,
as is common in CT problems.

The main goal of our work is to study the performance and the regularizing effects
of the AB- and BA-GMRES algorithms when applied to CT reconstruction problems.
To do this, we deliberately commit ``inverse crime'' and
generate the noise-free data as $\bar{b} = A\, \bar{x}$,
meaning that the same model
(the matrix $A$) is used to generate and reconstruct the data.
To determine which unmatched pair is preferable, one must use real (or carefully
simulated) data to avoid the inverse crime -- this is a topic for future research.

We remark that an alternative to using the unmatched transpose
is to use Kaczmarz's method as preconditioner
(i.e., for $B$) in the AB- and BA-GRMES algorithms,
as described and analyzed in \cite{MorikuniHayami2015,KaczmarzInner}.
We shall not pursue that approach here, but leave it for future work.

Our paper is organized as follows.
In Section \ref{sec:ABBA} we summarize the AB- and BA-GMRES algorithms, and
in Section \ref{sec:pert} we present new first-order perturbation theory for
the unmatched normal equations.
The behavior of the iterative methods in the presence of noisy data is
discussed in Section \ref{sec:iter}, and in Section \ref{sec:RegProp}
we study the regularizing properties of AB- and BA-GMRES when $B=A^\trans$.
Finally, in Section \ref{sec:NumEx} we present a number of numerical
examples that illustrate our theory and the performance of the AB- and
BA-GMRES algorithms.
We use the following notation: $\range(\cdot)$ denotes the range of a matrix, $\nul(\cdot)$
denotes the null space of a matrix.
For Krylov subspaces we use the notation
\[
\mathcal{K}_k(M,d) = \mathrm{span}\{ d, Md, M^2d, \ldots, M^{k-1}d \} \ .
\]
Moreover, $v|_{\mathcal{S}}$ denotes the orthogonal projection of vector $v$ on
the subspace ${\mathcal{S}}$, and $\sigma_i(A)$ denotes the $i$th singular
value of~$A$.

\section{The AB-GMRES and BA-GMRES Algorithms}
\label{sec:ABBA}

The two algorithms AB-GMRES and BA-GMRES were originally presented
and analyzed in \cite{HaYI10}
as alternative methods for solving linear least squares
problems~$\min_{x \in \mathbb{R}^n} \| b - A x \|_2$
according to the following principles:
\begin{itemize}
	\item
	AB-GMRES solves $\min_{y\in\mathbb{R}^m} \| b - \AB\, y \|_2$, $x = B \, y$
	with $B\in\mathbb{R}^{n\times m}$ as a right preconditioner.
	\item
	BA-GMRES solves $\min_{x\in\mathbb{R}^n} \| \Bb - \BA \, x \|_2$
	with $B\in\mathbb{R}^{n\times m}$ as a left preconditioner.
\end{itemize}
Here we briefly summarize these algorithms
using the notation $H_k = (h_{ij}) \in \mathbb{R}^{(k+1) \times k}$
and $e_1 = (1,0,\ldots,0)^\trans \in \mathbb{R}^{k+1}$;
we describe some stopping rules later in Section~\ref{sec:sr}.

\medskip\noindent
\begin{tabbing}
	xxx \= xxx \= xxx \= xxxxxxxxxxxxxxxxxxxxxxxxxxxxxxx \= xxx \= xxx \= xxx \= xxx \kill
	\textbf{Algorithm AB-GMRES} \> \> \> \> \textbf{Algorithm BA-GMRES} \\
	\> Choose initial $x_0$ \> \> \> \> Choose initial $x_0 $ \\
	\> $r_0 = b - A\, x_0$ \> \> \> \> $r_0 = \Bb - \BA \, x_0$ \\
	\> $w_1 = r_0 / \| r_0 \|_2$ \> \> \> \> $w_1 = r_0 / \| r_0 \|_2$ \\
	\> for $k= 1,2,\ldots$ \> \> \> \> for $k= 1,2,\ldots$ \\
	\> \> $q_k = AB \, w_k$ \> \> \> \> $q_k = \BA \, w_k$ \\
	\> \> for $i=1,2,\ldots,k$ \> \> \> \> for $i=1,2,\ldots,k$ \\
	\> \> \> $h_{i,k} = q_k^\trans w_i$ \> \> \> \> $h_{i,k} = q_k^\trans w_i$ \\
	\> \> \> $q_k = q_k - h_{i,k} \, w_i$ \> \> \> \> $q_k = q_k - h_{i,k} \, w_i$ \\
	\> \> endfor \> \> \> \> endfor \\
	\> \> $h_{k+1,k} = \| q_k \|_2$ \> \> \> \> $h_{k+1,k} = \| q_k \|_2$ \\
	\> \> $w_{k+1} = q_k / h_{k+1,k}$ \> \> \> \> $w_{k+1} = q_k / h_{k+1,k}$ \\
	\> \> $y_k = \arg\min_y
	\left\| \: \| r_0 \|_2 \, e_1 - H_k \, y \, \right\|_2 $ \> \> \> \>
	$y_k = \arg\min_y
	\left\| \: \| r_0 \|_2 \, e_1 - H_k \, y \, \right\|_2$ \\
	\> \> $x_k = x_0 + B \, [ w_1,w_2,\ldots,w_k ] \, y_k$ \> \> \> \>
	$x_k = x_0 + [ w_1,w_2,\ldots,w_k ] \, y_k$ \\
	\> \> $r_k =b-A\,x_k$ \> \> \> \> $r_k=b-A\, x_k$ \\
	\> \> stopping rule goes here \> \> \> \> stopping rule goes here \\
	\> endfor \> \> \> \> endfor
\end{tabbing}

The following statements about the convergence are from \cite{HaYI10}.
We emphasize that the two methods use the same Krylov subspace $\mathcal{K}_k(\BA, \Bb)$
for the solution, but they use different objective functions.

\medskip\noindent
\textbf{AB-GMRES}
\begin{itemize}
	\item
	AB-GMRES applies GMRES to $\min_u \| b - AB\,u \|_2$, $x=B\, u$, i.e.,
	forms the iterates $u_k \in u_0 + \mathcal{K}_k(AB,b)$ that minimize
	$\| b - AB\,u \|_2$, i.e., $x_k \in x_0 + \mathcal{K}_k(\BA, \Bb)$
	that minimize $\| b - A\, x \|_2$.
	\item
	The equality~$\min_x \| b-A\, x \|_2 = \min_z \| b - AB\, z \|_2$ holds for all
	$b\in\mathbb{R}^m$
	if and only if $\range(AB) = \range(A)$ \cite[Theorem~3.1]{HaYI10}.
	Note that $\range(AB) = \range(A)$ holds if $\range(B) = \range(A^\trans)$
	\cite[Lemma~3.3]{HaYI10}.
	\item
	If $\range(B) = \range(A^\trans)$, then AB-GMRES determines
	a solution of $\min_x \| b-A\, x \|_2$ without breakdown for all
	$b \in \mathbb{R}^m$  if and only if $\range(B^\trans) = \range(A)$
	\cite[Corollary~3.8.]{HaYI10}.
\end{itemize}

\noindent
\textbf{BA-GMRES}
\begin{itemize}
	\item
	BA-GMRES applies GMRES to min $\| \Bb  - \BA \, x \|_2$, i.e., forms the iterates
	$x_k \in x_0 + \mathcal{K}_k(\BA, \Bb)$ that minimize
	$\| \Bb - \BA\,x \|_2$.
	\item
	The problems $\min_x \| b - A\, x \|_2$ and $\min_x \| \Bb - \BA\, x \|_2$
	are equivalent for all $b\in\mathbb{R}^m$ if and only if $\range(B^\trans \BA) = \range(A)$
	\cite[Theorem~3.11]{HaYI10}.
	Note that $\range(B^\trans \BA) = \range(A)$ holds if
	$\range(B^\trans) = \range(A)$ \cite[Lemma~3.14]{HaYI10}.
	\item
	If $\range(B^\trans)=\range(A)$ and $\range(B) \cap \mathcal{N}(A) = \{ 0 \}$, then
	BA-GMRES determines a solution of $\min_x \| b - A\, x \|_2$ without breakdown
	for all $b \in \mathbb{R}^m$ \cite[Theorem~3.1]{MorikuniHayami2015}.
\end{itemize}

The conditions on the ranges $\range(AB) = \range(A)$ and
$\range(B^\trans \BA) = \range(A)$ ensure the equivalence between the
matched and unmatched normal equations, as seen above.
Simply, $\range(B) = \range(A^\trans)$ implies
$\range(AB) = \range(A)$ and
$\range{R}(B^\trans) = \range(A)$ implies
$\range(B^\trans \BA) = \range(A)$.
Further, these conditions also ensure the convergence of AB- and BA-GMRES\@.
However, it is not easy to check if these conditions are satisfied in practice.

If $B \approx A^\trans$, solving the unmatched normal equations by using AB-
and BA-GMRES we expect to obtain an approximation to the solution.
This expectation can be supported by the observation that if $B$ is close to
$A^\trans$, then $\range(B)$ and $\range(B^\trans)$ can be close to
$\range(A^\trans)$ and $\range(A)$, respectively, e.g.,
the principal angles between the pair of the ranges $\range(B)$ and
$\range(A^\trans)$ and pair $\range(B^\trans)$ and
$\range(A)$ can be small from the extended $\sin \theta$
theorem~\cite[Section~3]{Wedin1972BIT}, \cite[Chapter~2]{StewartSun1990}
\begin{equation}
	\left( {\left\| \sin \theta \bigl( \range(B), \range(A^\trans) \bigr)
		\right\|}^2 + {\left\| \sin \theta \bigl( \range(B^\trans), \range(A) \bigr)
		\right\|}^2 \right)^{1/2} \leq
	\frac{ \sqrt{2} \, \| B - A^\trans \|}{\max\bigl(\sigma_r(A),\sigma_r(B\bigr))} \ ,
\end{equation}
where $\| \cdot \|$ is either the 2-norm or the Frobenius norm and
we assume that $r = \mathrm{rank}(A) = \mathrm{rank}(B)$.
Moreover, $ \sin \theta(\mathcal{S},\mathcal{T}) = \mathrm{diag}( \sin \theta_1, \sin \theta_2,\dots, \sin \theta_r)$
is a diagonal matrix with the principal angles
$\theta_1 \geq \theta_2 \geq \cdots \geq \theta_r$ between subspaces
$\mathcal{S}$ and $\mathcal{T}$ with $r = \dim \mathcal{S} = \dim \mathcal{T}$.
In case of the matrix 2-norm, we can instead use the upper bound
\[ \sqrt{2} \min \left(\, \kappa_2(A) \, \frac{\| B - A^\trans \|_2}{\| A \|_2},
\kappa_2(B) \, \frac{\| B - A^\trans \|_2}{\| B \|_2} \right) .
\]

\section{First-Order Perturbation Theory}
\label{sec:pert}

To further motivate the use of the unmatched normal equations,
we consider the difference between the solutions of
the matched and unmatched normal equations, and we study how this
difference depends on the unmatchness of the pair $(A,B)$.

A first-order perturbation analysis for the unmatched normal equations
was given in \cite[Section~2.1]{ElHa18}.
Their analysis gave a bound for the distance between the closest pair of a point
in the solution set of the least squares problem~
$\min_{x \in \mathbb{R}^n} \| b - A \, x \|_2$ and a point in the solution set
of its perturbed problem.
Here, we perform an alternative first-order perturbation analysis specifically
for the minimum-norm solutions, as discussed below.
This analysis refines the previous perturbation analysis in \cite[Section~2.1]{ElHa18}.

Generally, we can assume to have modeling error also in $A$, in addition to $B$.
Let $A$ and $\bar{b}$ be the ideal model and data, respectively, and let
\begin{align}
	\tilde{A} = A + E_1, \quad \tilde{B} = A^\trans + E_2^\trans, \quad
	b = \bar{b} + \delta b
\end{align}
be the perturbed models and data.
Moreover, let $\bar{x}$ be a ground truth solution of
$\min_{x \in \mathbb{R}^n} \| b - A x \|_2$ and
$\bar{r} = \bar{b} - A \bar{x}$ be the corresponding least squares residual,
for which $A^\trans \bar{r} = 0$ holds.
The perturbed matrices $\tilde{A}$ and $\tilde{B}$ can be regarded as $A$ and $B$
in the unmatched normal equations~\eqref{eq:unmatchednormalequations}.

Now, consider solving the perturbed unmatched normal equations 
\begin{align}
	\tilde{B} \hspace{1.9pt} \tilde{A} (x_{\min} + \delta x_{\min}) =
	\tilde{B} \hspace{1.9pt} b \ ,
	\label{eq:unmatched_ne}
\end{align}
where $x_{\min}$ denotes the minimum-norm solution of the
unperturbed normal equations $A^\trans A \, x = A^\trans b$.
Note that the linear system~\eqref{eq:unmatched_ne} is consistent if
$\tilde{B} b \in \mathcal{R}(\tilde{B} \tilde{A})$.
Irrespective of the inconsistency of \eqref{eq:unmatched_ne},
the minimum-norm solution of the least squares problem
\begin{align}
	\min_{x_{\min} + \delta x_{\min} \in \mathbb{R}^n}
	\| \tilde{B} \hspace{1.9pt} b - \tilde{B} \hspace{1.9pt}\tilde{A} \,
	(x_{\min} + \delta x_{\min}) \|_2
	\label{eq:trans_unmatched_ne}
\end{align}
is given by $(\tilde{B} \hspace{1.9pt} \tilde{A})^\dag \tilde{B} \hspace{1.9pt} b$,
where $\dagger$ denotes the Moore-Penrose generalized inverse (pseudoinverse).
We are concerned with the difference
between the minimum-norm solution of $\min_{x \in \mathbb{R}^n} \| b - A x \|_2$
and the minimum-norm solution $x_\mathrm{min} + \delta x_\mathrm{min}$
of \eqref{eq:trans_unmatched_ne}.
Note that the solution of interest in \eqref{eq:Axb}
is not necessarily the minimum-norm solution
but may lie close to it.

\begin{theorem} \label{th:perturb_bound}
	Assume that $\tilde{A}$ and $\tilde{B}$ are both acute perturbations of $A$
	and $A^\trans$, respectively, i.e.,
	\[
	\| P_{\mathcal{R}(\tilde{A})} - P_{\mathcal{R}(A)} \|_2 < 1 \ , \quad
	\| P_{\mathcal{R}(\tilde{A}^\trans)} - P_{\mathcal{R}(A^\trans)} \|_2 < 1
	\]
	and
	\[
	\| P_{\mathcal{R}(\tilde{B})} - P_{\mathcal{R}(A^\trans)} \|_2 < 1 \ , \quad
	\| P_{\mathcal{R}(\tilde{B}^\trans)} - P_{\mathcal{R}(A)} \|_2 < 1 \ ,
	\]
	respectively, where $P_\mathcal{S}$ denotes the orthogonal projection onto a
	subspace $\mathcal{S}$.
	Then, the first-order bound of the relative error norm is given by
	\begin{align}
		\frac{\| \delta x_{\min} \|_2}{\| x_{\min} \|_2}
		\leq \kappa_2 (A) \left[ \sigma_r^{-1} \left( 2 \| E_1 \|_2
		\frac{\| \bar{b}|_{\mathcal{R}(A)} \|_2}{\| \bar{b} \|_2} + \| E_2 \|_2
		\frac{\| \bar{b}|_{\mathcal{R}(A)^\perp} \|_2}{\| \bar{b} \|_2} \right) +
		\frac{\| \delta b|_{\mathcal{R}(A)} \|_2}{\| \bar{b} \|_2} \right],
	\end{align}
	where $\sigma_r$ denotes the smallest nonzero singular value of $A$
	and $\kappa_2(A) = \| A \|_2/\sigma_r$ is the condition number of $A$.
\end{theorem}

See \ref{app:proof4perturb_bound} for the proof.
This theorem shows that the error bound depends linearly on
$\| E_1 \|$, $\| E_2 \|$, $\| \bar{b}|_{\mathcal{R}(A)} \|$,
$\| \bar{b}|_{\mathcal{R}(A)^\perp} \|$
and $\| \delta b |_{\range(A) \|_2}$, whereas
the bound is independent of $\delta b|_{\mathcal{R}(A)^\perp}$.
If the smallest nonzero singular value~$\sigma_r$ is very small,
then the perturbations~$E_1$ and $E_2$ can greatly affect the error.

The first-order bound of the relative error norm in the ``inverse crime'' case,
where $\bar{b} = A\,\bar{x}$ with the minimum-norm solution $\bar{x} = x_{\min}$
and hence $\bar{r}=0$, is given as follows.

\begin{corollary} \label{cor:perturb_bound}
	Assume that $E_1 = 0$, $\tilde{B}$ is an acute perturbations of
	$A^\trans$, and $\bar{b}|_{\mathcal{R}(A)^\perp} = 0$.
	Then, the first-order bound of the relative error norm is given by
	\begin{align}
		\frac{\| \delta x_{\min} \|_2}{\| x_{\min} \|_2} \leq \kappa_2(A)
		\frac{\| \delta b|_{\mathcal{R}(A)} \|_2}{\| \bar{b} \|_2}.
	\end{align}
\end{corollary}

This corollary follows directly from Theorem~\ref{th:perturb_bound}
and shows that the error bound is independent of $E_2$ and
$\delta b|_{\mathcal{R}(A)^\perp}$ in the ``inverse crime'' case.
Note that the higher-order terms of these quantities,
such as $E_2^T \delta b$, can contribute to the error.

The above analysis focuses on the unmatched normal equations that
BA-GMRES deals with.
A first-order perturbation analysis for the unmatched normal
equations~$\tilde{A} \tilde{B} y = b$, $x = \tilde{B} y$ that AB-GMRES
deals with in the consistent case~$b \in \mathcal{R}(\tilde{A}
\tilde{B})$ is performed in \cite[section~2.2]{ElHa18}.
The corresponding analysis in the inconsistent case~$b \not \in
\mathcal{R}(\tilde{A} \tilde{B})$ is left open.

\section{Iterative Regularization and Semi-Convergence}
\label{sec:iter}

When we discretize an inverse problem we obtain a coefficient matrix $A$
whose nonzero singular values decay gradually to zero with no gap anywhere,
and $A$ has a large condition number.
Therefore, it is not a good idea to naively solve the problem \eqref{eq:Axb}
with noisy data.
With the notation from \eqref{eq:Axb} and
assuming that the exact solution $\bar{x}$ satisfies
\begin{equation}
	\bar{x} = A^{\dagger} \bar{b} \in \range(A^\trans) \ ,
\end{equation}
the minimum-norm least squares solution to the noisy problem has the form
$A^{\dagger}b = \bar{x} + A^{\dagger}e$.
Here, the second term $A^{\dagger}e$ is highly undesired because -- due to the
large condition number of $A$ -- it has elements that are much
larger than those in~$\bar{x}$.
We need to use a regularization method that filters the influence from the noise.

The singular value decomposition (SVD) provides a convenient framework for
analyzing this situation.
Let the coefficient matrix $A$ in \eqref{eq:Axb} have the SVD
\begin{equation}
	\label{eq:SVD}
	A = U\, \Sigma\, V^\trans = \sum_{i=1}^r u_i \, \sigma_i \, v_i^\trans \ , \qquad
	\sigma_1 \geq \sigma_r \geq \cdots \geq \sigma_r > 0 \ , \qquad
	r = \mathrm{rank}(A)
\end{equation}
with
\begin{equation}
	U = [ \, u_1, \, u_2, \, \ldots, u_r \, ] \in \mathbb{R}^{m\times r} \ , \qquad
	\Sigma = \mathrm{diag}(\sigma_i) \in \mathbb{R}^{r\times r} \ , \qquad
	V = [ \, v_1, \, v_2, \, \ldots, v_r \, ] \in \mathbb{R}^{n\times r} \ .
\end{equation}
Then we can write the minimum-norm least squares solution as
\begin{equation}
	\label{eq:mnsol}
	A^{\dagger}b = \bar{x} + A^{\dagger}e =
	\sum_{i=1}^r \frac{u_i^\trans\bar{b}}{\sigma_i} v_i +
	\sum_{i=1}^r \frac{u_i^\trans e}{\sigma_i} v_i \ .
\end{equation}
Discretizations of inverse problems satisfy the \emph{discrete Picard condition} (DPC)
meaning that, in average, the absolute values $|u_i^\trans\bar{b}|$ decay
faster than the singular values \cite[Section~3.3]{Hansen}.
The first term in \eqref{eq:mnsol} is equal to the exact, noise-free solution
$\bar{x} = A^{\dagger} \bar{b}$
and the DPC ensures that its norm stays bounded as the problem size increases.
If $e$ is white noise then so is $U^\trans e$ meaning that the coefficients
$u_i^\trans e/\sigma_i$ will, on average, increase due to the decreasing singular values.
Consequently, the second term in \eqref{eq:mnsol}
is typically much larger than the first term -- and its
norm increases with the problem size.
All regularization methods essentially filter or dampen the SVD components
corresponding to the smaller singular values, thus reducing the influence
of the noise and, at the same time, computing a good approximation to~$\bar{x}$.

The key mechanism behind the use of \emph{iterative solvers} for computing
solutions to inverse problems with noisy data, such as the CT reconstruction problem,
is known as \emph{semi-convergence} \cite[Chapter~6]{Hansen}.
When we apply an iterative method (typically a least squares solver) to
the noisy problem \eqref{eq:Axb} then the reconstruction
error $\| \bar{x}-x_k \|_2$, where $x_k$ denotes the $k$th iterate,
exhibits two different phases:
\begin{enumerate}
	\item
	During the initial iterations $\| \bar{x}-x_k \|_2$ decreases and
	$x_k$ appears to approach the ground truth~$\bar{x}$.
	\item
	After a while $\| \bar{x}-x_k \|_2$ starts to increase, and
	asymptotically $x_k$ converges to the undesired least squares solution.
\end{enumerate}
To obtain a meaningful regularized solution we must stop the iterations at
the transition point where the iteration vector $x_k$ is as close to
$\bar{x}$ as possible.
Development of stopping rules that seek to terminate the iterations at this point
are closely related to methods for choosing regularization parameters
(see, e.g., \cite[Chapter~5]{SIAMbook} and \cite{ReRo13});
overviews of stopping rules in the context of CT
are given in \cite[Section~11.2]{SIAMbook} and \cite{HaJR21}.

Deeper insight into the semi-convergence, and explanations when and why
it manifests itself, has been a topic of research for many years.
For methods where $x_k$ can be expressed as a filtered SVD solution of the form
\begin{equation}
	\label{eq:fSs}
	x_k = \sum_{i=1}^r \phi_i^{(k)} \, \frac{u_i^\trans b}{\sigma_i} v_i \ ,
	\qquad \phi_i^{(k)} = \hbox{filter factor at $k$th iteration} \ ,
\end{equation}
we have a good understanding, see, e.g., \cite{SIAMbook} and
\cite[Chapter~6]{Hansen}.
For example, for the Landweber iteration
$x_k = x_{k-1} + \omega\, A^\trans (b-A\, x_{k-1})$ with relaxation parameter $\omega$
we have $\phi_i^{(k)} =  1 - (1 - \omega\,\sigma_i^2)^k$, and for CGLS
the filters $\phi_i^{(k)}$ can be expressed as polynomials of degree $k$
that depend on the Ritz values associated with the $k$th CGLS iteration,
cf.\ \cite[Section~6.4.3]{DIPbook}.

For other methods where $x_k$ cannot easily be expressed in terms of the SVD
the understanding is less mature (see  \cite{ElHN14}, \cite{LiHH21} for
emerging insight into Kaczmarz's method).
The analysis of the regularizing properties of GMRES applied to $A\, x=b$
is complicated by
the fact that it is connected to the convergence of the Ritz values
of the underlying Arnoldi algorithm.
\begin{itemize}
	\item
	The insight obtained from \cite{CaLR02} is that
	if the noise-free data $\bar{b}$ lies in a finite-dimensional Krylov subspace,
	and if GMRES is equipped with a suitable stopping rule,
	then the GMRES-solution converges to the exact solution $\bar{x}$
	as the noise goes to zero.
	\item
	The focus of \cite{GaNo16} is so-called ``hybrid methods'' where regularization
	is applied to the Hessenberg systems in GMRES, but the main result in
	\cite[\S 3.1.2]{GaNo16} applies more generally:
	Assume that the system in \eqref{eq:Axb} satisfies the DPC and that
	the left singular vectors of the Hessenberg matrices of two consecutive GMRES steps,
	applied to $A\, x=b$, resemble each other --
	then the Hessenberg systems in GMRES also satisfy the DPC.
\end{itemize}
Taken together, these results imply that if all the SVD components corresponding
to the larger singular values are captured in order of decreasing
magnitude, when GMRES is applied to $A\, x=b$,
then GMRES will exhibit semi-convergence.
Unfortunately, a complete understanding of these aspects has not emerged yet.
Hence, while the semi-convergence aspect of GMRES is crucial in this work,
we primarily rely on insight obtained from numerical experiments.

\section{The Regularizing Properties of AB- and BA-GMRES with a Matched Transpose}
\label{sec:RegProp}

To understand the regularizing properties of the AB- and BA-GMRES methods
when $B \approx A^\trans $, let us consider the limiting case when $B=A^\trans$
(the matched case).

\subsection{The AB-GMRES Algorithm with $B=A^\trans$}
The $k$th step of AB-GMRES with $B=A^\trans$ solves
\[
\min_{u \in u_0 + {\cal K}_k ( AA^\trans ,r_0 )}
{ \| b - A \, A^\trans u \|_2 }^2
\]
and it determines the $k$th iterate where $x_k = A^\trans u_k$. Hence,
\[
x_k = A^\trans u_k \in A^\trans u_0+ A^\trans {\cal K}_k (AA^\trans, r_0)
= x_0 +{\cal K}_k (A^\trans A, A^\trans r_0) \ .
\]
The method minimizes
\[
{ \| r_k \|_2 }^2 = { \| r_k |_{\range(A)} \|_2 }^2 +
{ \| r_k |_{ {\range(A) }^\perp } \|_2 }^2 \ ,
\]
where
\[
r_k = b - A x_k = b|_{\range(A)} + b|_{\range(A)^\perp} - A \, x_k
= r_k|_{\range(A)} + b|_{\range(A)^\perp}
\]
and
\[
r_k|_{\range(A)} = b|_{\range(A)} - A x_k \ , \qquad
r_k|_{\range(A)^\perp}=b|_{\range(A)^\perp} \ .
\]
Hence, the method minimizes ${ \| r_k|_{\range(A)} \|_2 }^2$.
In summary, the AB-GMRES method with $B=A^\trans$ minimizes
${ \| r_k|_{\range(A)} \|_2 }^2$, and its iterates satisfy
$ x_k \in x_0 + {\cal K}_k (A^\trans A, A^\trans r_0)$.

\subsection{The LSQR Algorithm}

Next, consider the LSQR algorithm for $\min_{x} \| b - A x\|_2$
which is mathematically equivalent to the CGLS method.
Note
\begin{equation}
	\label{NR}
	\min_{x} \| b - A \, x \|_2 \qquad \Longleftrightarrow \qquad A^\trans A \, x = A^\trans b
	\qquad \Longleftrightarrow \qquad A^\trans r = 0 \ ,
\end{equation}
where $r = b - A \, x$.
The LSQR and CGLS methods are mathematically equivalent to applying
the Conjugate Gradient (CG) method to the normal
equations $A^\trans A \, x = A^\trans b$.
They minimize $\be^\trans A^\trans A \be = (A\be)^\trans A\be = { \|A\be\|_2 }^2$,
where $\be = x-x^\ast$ and $x^\ast$ is any solution of \eqref{NR}. Note
\[
A\be=A(x-x^\ast)=Ax-Ax^\ast=(b-Ax^\ast)-(b-Ax) =r^\ast - r \ ,
\]
where
\[
r^\ast=b-A x^\ast, \qquad r=b-Ax \ .
\]
Also note
\[
r = b - A x 
= r|_{\range(A)} + r|_{\range(A)^\perp} \ ,
\]
where
\[
r|_{\range(A)}=b|_\range(A) -Ax^\ast, \qquad
r|_{\range(A)^\perp}=b|_{\range(A)^\perp} \ ,
\]
and
\[
r^\ast = b - A x^\ast 
= r^\ast|_{\range(A)} + r^\ast|_{\range(A)^\perp} \ ,
\]
where
\[
r^\ast|_{\range(A)}=b|_\range(A) - A \, x^\ast,\qquad
r^\ast|_{\range(A)^\perp}=b|_{\range(A)^\perp} \ .
\]
Hence,
\[
0 = A^\trans r^\ast = A^\trans \Bigl( r^\ast|_{\range(A)} +
r^\ast|_{\range(A)^\perp} \Bigr) = A^\trans r^\ast|_{\range(A)} \ ,
\]
since $ \range(A)^\perp = \nul(A^\trans)$. Hence, 
\[
0 = r^\ast|_{\range(A)} \in \nul(A^\trans) \cap \range(A) =
\range(A)^\perp \cap \range(A) = \{ 0 \} \ .
\]
Thus we have
\[
A\be=r^\ast-r=r^\ast|_{\range(A)} -r|_{\range(A)} = -r|_{\range(A)}
\]
and
\[
{ \|A \, \be\|_2 }^2 = { \|r|_{\range(A)} \|_2 }^2 \ .
\]
Thus, CGLS minimizes ${ \| r_k|_{\range(A)} \|_2 }^2 $, where
$ r_k = b - A x_k$. The iterates satisfy
$ x_k \in x_0 + {\cal K}_k (A^\trans A, A^\trans r_0)$.

Therefore, AB-GMRES with $B=A^\trans$, as well as LSQR and CGLS for
$ \min_x { \| b-Ax\|_2 }^2$, minimize ${ \|r_k|_{\range(A)}\|_2 }^2$,
where $r_k = b - A \, x_k$, and the iterates (solutions) are in the same space,
i.e., $x_k \in x_0 + {\cal K}_k (A^\trans A, A^\trans r_0)$.
Thus, these methods are mathematically equivalent.

In finite precision arithmetic, AB-GMRES with $B=A^\trans$ should be numerically
more stable than CGLS and LSQR, since AB-GMRES is based on the Arnoldi process
whereas LSQR and CGLS rely on short-term recurrences.
In fact, AB-GMRES may be numerically equivalent to LSQR and CGLS with full
reorthogonalization \cite{HaYI10},
and this is confirmed by our numerical experiments (which are not included here).

Now, it is well known that CGLS has good regularizing properties for discrete ill-posed
problems, leading to semi-convergence, see, e.g. \cite{Hansen,MHC,MHL,JH}.
When $B \approx A^\trans$, we may still apply AB-GMRES,
while applying LSQR (which is equivalent to applying CG to
the unmatched normal equations $\BA\, x = \Bb$) is not well founded
and may be problematic since $\BA$ is neither symmetric nor positive semi-definite.
Also, we may still expect semi-convergence of AB-GMRES, as will be demonstrated
in the numerical experiments.
Specifically, in Section \ref{sec:unmatchedness} we study experimentally
how semi-convergence is influenced by the difference between $B$ and $A^\trans$.

\subsection{The BA-GMRES Algorithm with a Matched Transpose}

BA-GMRES applies GMRES to $\min_x \| \Bb - \BA \, x\|_2$.
It minimizes ${ \|Br_k\|_2 }^2$ where $r_k=b-Ax_k$, and
\[
x_k \in x_0 + {\cal K}_k (\BA, Br_0) = x_0 + B\, {\cal K}_k (AB,r_0) \ .
\]
BA-GMRES with $B=A^\trans$ applies GMRES to
$\min_x \| A^\trans b - A^\trans Ax\|_2$.
It minimizes
\[
{ \|A^\trans r\|_2 }^2 = (A^\trans r)^\trans A^\trans r=r^\trans AA^\trans r \ ,
\]
where
\[
r = b-Ax = r|_{\range(A)} + r|_{\range(A)^\perp} \qquad \hbox{and} \qquad
r|_{\range(A)} = b_{\range(A)}-Ax \ , \quad
r|_{\range(A)^\perp} = b|_{\range(A)^\perp} = b|_{\nul(A^\trans)} \ .
\]
Thus,
\[
A^\trans r=A^\trans \Bigl( r|_\range(A) + b|_{\nul(A^\trans )} \Bigr)
=A^\trans r|_{\range(A)}
\]
and BA-GMRES with $B=A^\trans$ minimizes ${ \|A^\trans r|_{\range(A)}\|_2 }^2$.
The iterates satisfy
\[
x_k \in x_0 + {\cal K}_k (A^\trans A, A^\trans r_0)
= x_0 + A^\trans {\cal K}_k (AA^\trans, r_0) \ .
\]

BA-GMRES with $B=A^\trans$ applies GMRES to
\[
A^\trans Ax=A^\trans b \qquad \Longleftrightarrow \qquad
\min_{x\in \mathbb{R}^n} \| b - A x \|_2 \ ,
\]
which is mathematically equivalent to applying MINRES to the normal equations
$A^\trans A x = A^\trans b $, and which is equivalent to LSMR \cite{FS}.
Again, BA-GMRES with $B=A^\trans$ should be numerically more stable than LSMR\@.
It may be numerically equivalent to LSMR with full reorthogonalization, and
again our numerical experiments (not included here) confirm this.

Similar to CGLS, LSMR also has good regularizing properties when applied to
discrete ill-posed problems \cite{FS,J}.
We may still apply BA-GMRES when $B \approx A^\trans$, while it is not well founded
to apply LSMR since $\BA$ is no longer symmetric and we cannot apply MINRES
to $\BA\,x = \Bb$.
We expect that BA-GMRES exhibits semi-convergence, as will be demonstrated
in the numerical experiments.

\section{Numerical Examples}
\label{sec:NumEx}

In this section we illustrate the use of the AB- and BA-GMRES methods
for CT problems with an unmatched transpose.
We start with numerical results related to the eigenvalues
of the iterations matrices $AB$ and $\BA$, and then we demonstrate
the semi-convergence of the methods.
All computations are performed in MATLAB using our own implementations
of AB- and BA-GMRES which are available from us,
LSQR is from Regularization Tools \cite{Regtools} and LSMR is from
MathWorks' File Exchange \cite{LSMR}.

\subsection{The Test Problems}

\begin{table}
	\caption{\label{table:parameters}The parameters used to generate the small and large
		test matrices; the matrix dimensions are $m=N_{\mathrm{ang}}N_{\mathrm{det}}$
		and $n=N^2$.  We also show the ground truth $\bar{x}$ as an image.}
	\medskip
	\scriptsize
	\centering
	\renewcommand{\arraystretch}{1.3}
	\begin{tabular}{|l|c|c|c|} \hline
		Parameter & Small matrix & Large matrix & Ground truth \\ \hline
		Image size $N\times N$ & $128 \times 128$ & $420 \times 420$ &
		\multirow{6}{*}{\includegraphics[width=0.135\textwidth]{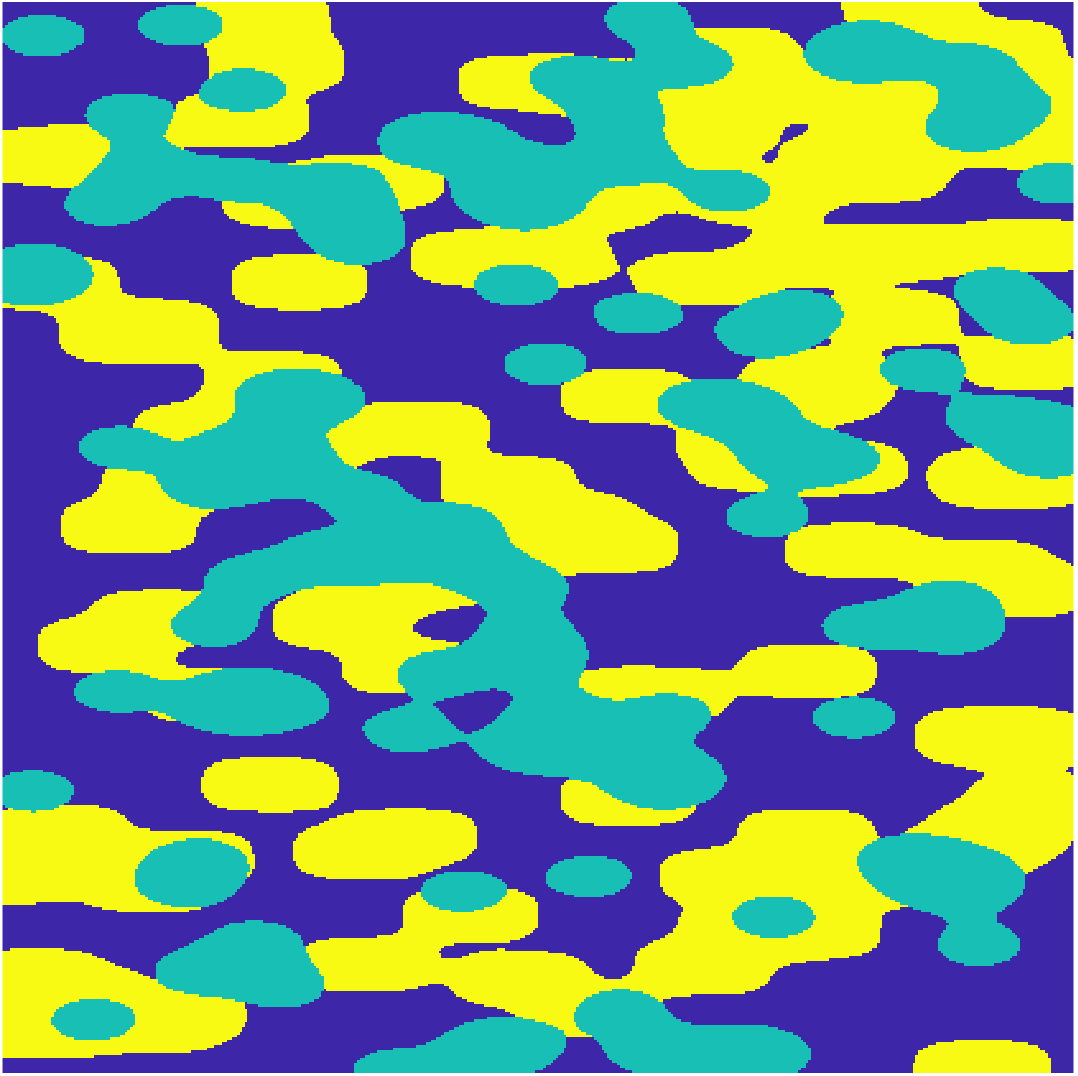}} \\
		Projection angles & $0^\circ,1^\circ,2^\circ,\ldots,179^\circ$ &
		$0^\circ,0.3^\circ,0.6^\circ,\ldots,179.7^\circ$ & \\
		No.\ projection angles $N_{\mathrm{ang}}$ & 180 & 600 & \\
		No.\ detector elements $N_{\mathrm{det}}$ & 128 & 420 & \\
		Matrix size $m\times n$ & $23\,040 \times 16\,384$ & $252\,000 \times 176\,400$ & \\
		Sparsity & $\approx$ 99\% & $\approx$ 99.6\% & \\ \hline
	\end{tabular}
\end{table}

\begin{table}
	\caption{\label{table:diffAA}The relative norm-wise differences between
		the three test matrices from ASTRA.}
	\medskip
	\scriptsize
	\centering
	\renewcommand{\arraystretch}{1.3}
	\begin{tabular}{|l|ccc|} \hline
		& $\| \As - \Al \|\fro / \| \As \|\fro$ & $\| \As - \Ai \|\fro / \| \As \|\fro$
		& $\| \Al - \Ai \|\fro / \| \Al \|\fro$ \\ \hline
		small matrices & 0.3700 & 0.1402 & 0.2648 \\
		large matrices & 0.3747 & 0.1405 & 0.2685 \\ \hline
	\end{tabular}
\end{table}

The matrices $A$ and $B$ used in these experiments are
representative of the matrices in many CT software
packages and applications.
They are generated by means
of the CPU version of the software package ASTRA \cite{ASTRA};
the matrices in the GPU version are not explicitly available
when using this software.
Three different discretization models are provided in ASTRA:
the line model, the strip model, and the interpolation (or Joseph) model;
see \cite[Chapter~9]{SIAMbook} for details.
We can then use any of the corresponding matrices
$\Al$, $\As$ and $\Ai$ to generate unmatched pairs $(A,B)$.
We use a parallel-beam geometry and two different sizes of these matrices
corresponding to the parameters listed in Table~\ref{table:parameters},
while Table~\ref{table:diffAA} lists the relative norm-wise differences
between these matrices.
The exact solution $\bar{x}$ is generated by means of the function
\begin{center}
	\texttt{phantomgallery('threephases',N)}
\end{center}
from \cite{AIRtoolsII}, and it is shown in Table~\ref{table:parameters}.
We add white Gaussian noise $e$ to $\bar{b}$ with two different relative
noise level $\| e \|_2 / \| \bar{b} \|_2 = 0.003$ and $0.03$.

\subsection{Eigenvalues}
\label{sec:eigenvalues}

\begin{figure}
	\centering
	\includegraphics[width=0.44\textwidth]{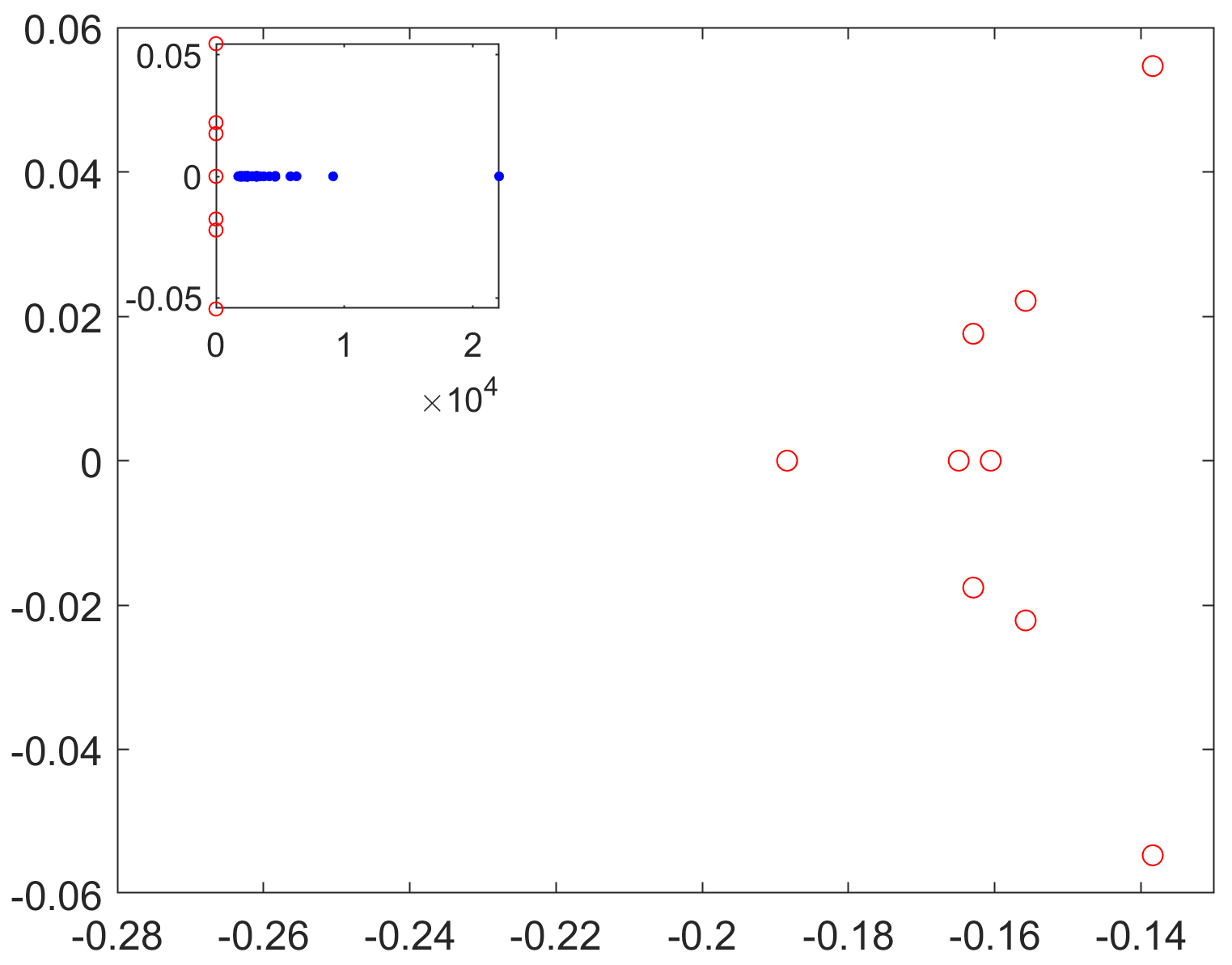} \hspace{2mm}
	\includegraphics[width=0.44\textwidth]{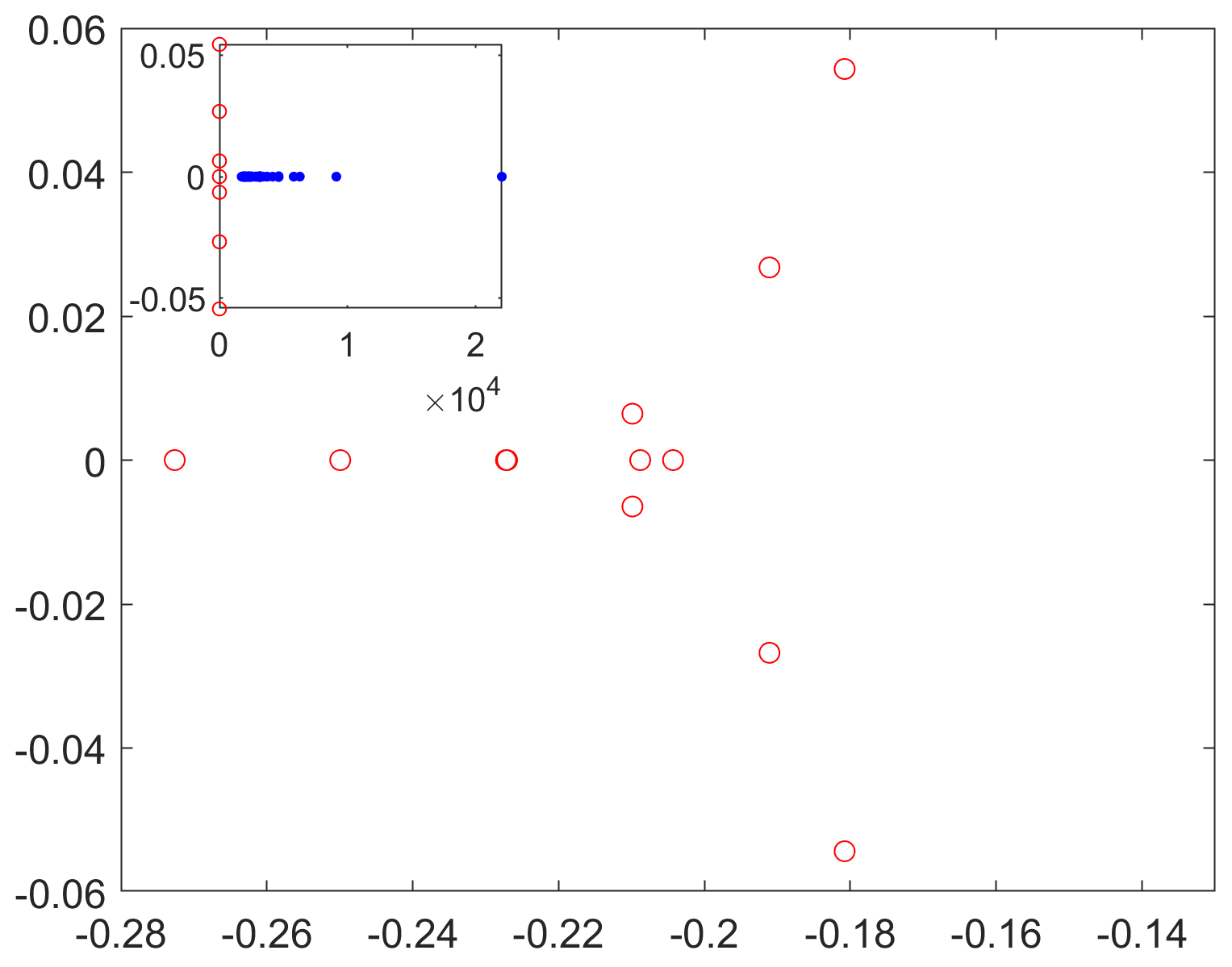} \\[2mm]
	\includegraphics[width=0.45\textwidth]{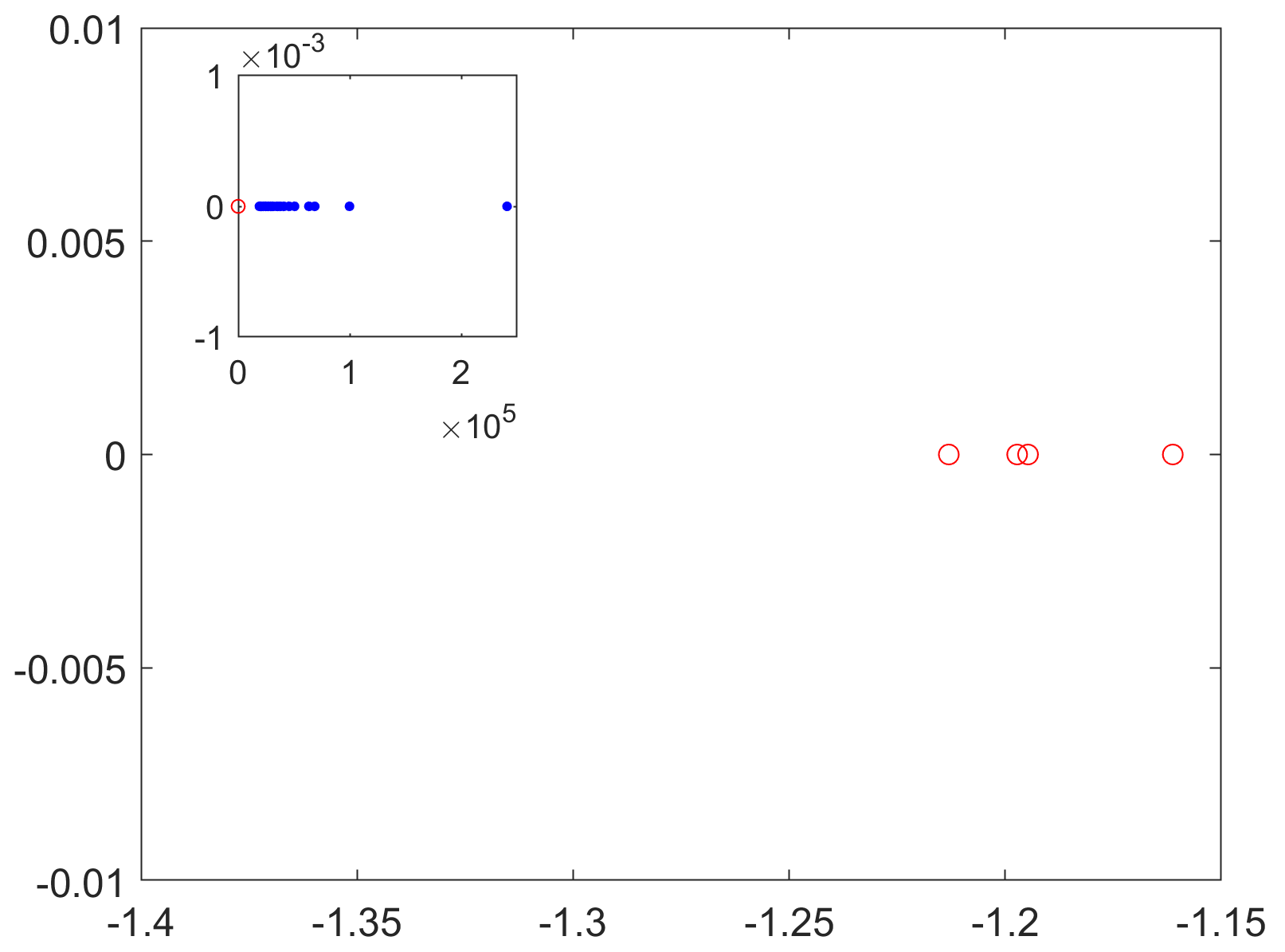} \hspace{2mm}
	\includegraphics[width=0.45\textwidth]{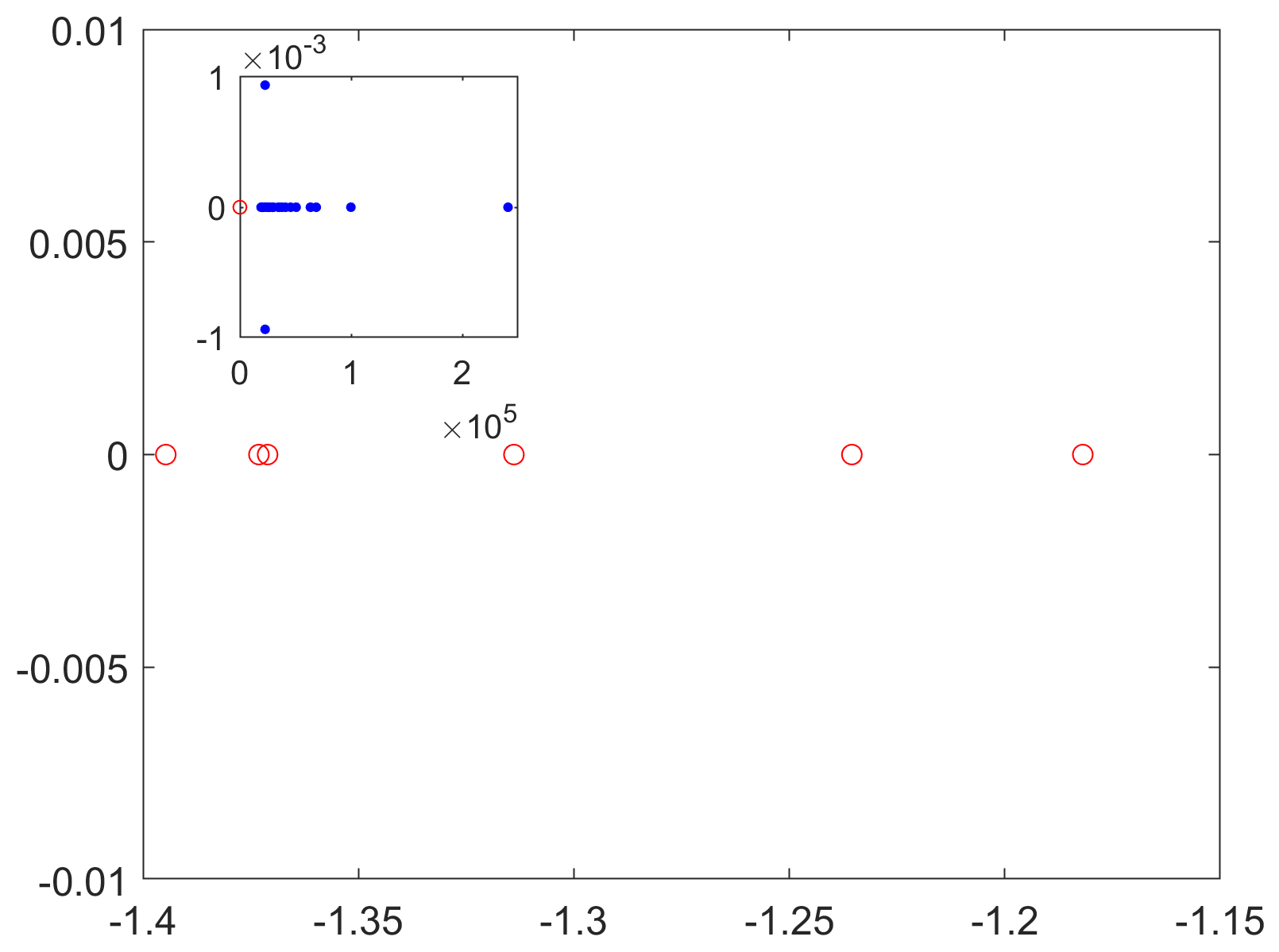}
	\caption{Selected eigenvalues of the matrices $\AlT\As$ (left) and $\AlT\Ai$ (right),
		for both the small (top) and large (bottom) test matrices.
		The red circles show the leftmost eigenvalues all of which have a negative real part.
		The inset figure shows these eigenvalues together with the 50 largest eigenvalues
		shown as blue dots.}
	\label{fig:eigenvalues}
\end{figure}

As mentioned in the Introduction, the standard SIRT iterative methods will not
converge when $\BA$ has complex eigenvalues with negative real part.
To demonstrate that this is the case for the discretizations used here,
Figure~\ref{fig:eigenvalues} shows the leftmost and the largest eigenvalues of $\BA$
computed by means of MATLAB's \texttt{eigs},
for two combinations of the small and large test matrices, namely,
$\AlT\As$ and $\AlT\Ai$.
In both cases there are indeed eigenvalues with negative real parts,
meaning that the SIRT methods do not converge.
Hence, it is natural to use the AB- and BA-GMRES methods for these matrices.
The leftmost eigenvalues of the third combination $\AsT\Ai$ have tiny negative
real parts.

\subsection{Error Histories}
\label{sec:eh}

\begin{figure}
	\centering
	$(\As,\AiT)$ \qquad\qquad\qquad Small matrices \qquad\qquad\qquad $(\As,\AlT)$ \\
	\includegraphics[width=0.45\textwidth]{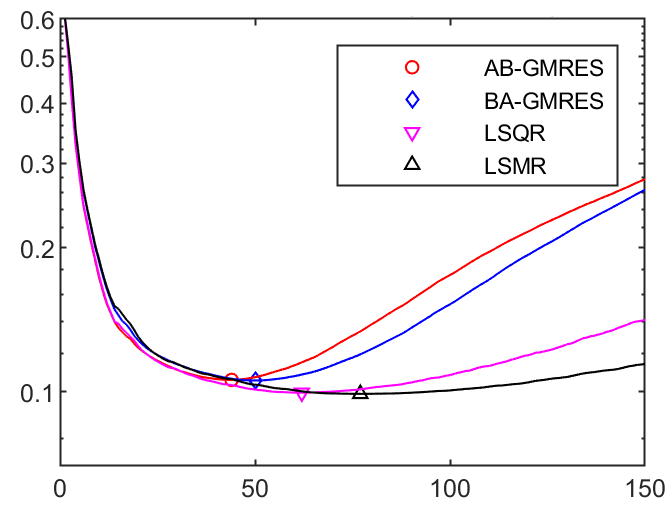} \quad
	\includegraphics[width=0.45\textwidth]{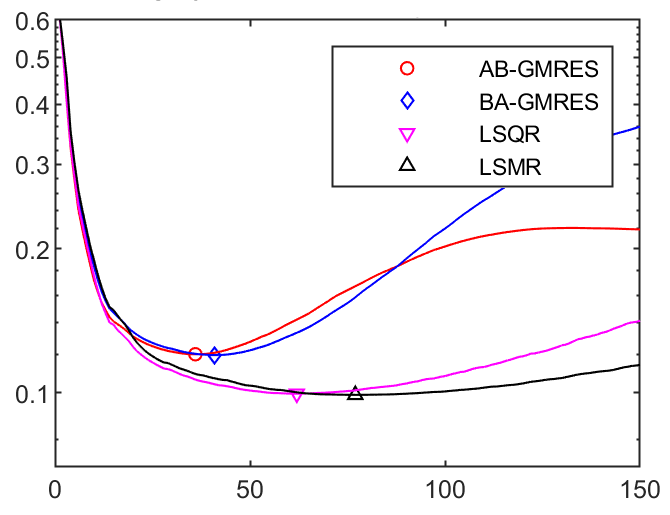} \\[3mm]
	$(\As,\AiT)$ \qquad\qquad\qquad Large matrices \qquad\qquad\qquad $(\As,\AlT)$ \\
	\includegraphics[width=0.45\textwidth]{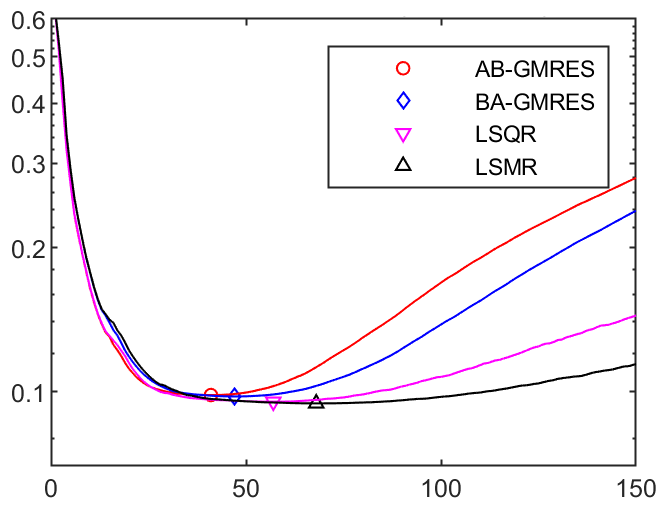} \quad
	\includegraphics[width=0.45\textwidth]{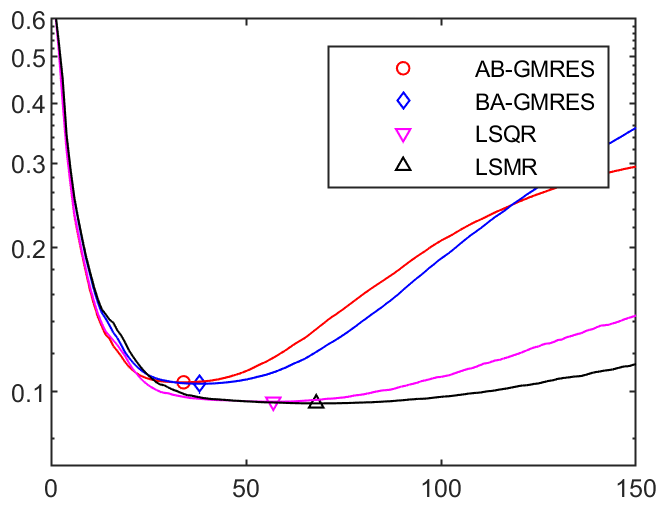}
	\caption{Error histories, i.e., plots of the relative reconstruction error
		$\|\bar{x}-x_k\|_2/\|\bar{x}\|_2$ versus the number of iterations.
		The minima are indicated by the markers.
		We show results for the cases $(A,B) = (\As,\AiT)$ in the left plots,
		and $(A,B) = (\As,\AlT)$ in the right plots.
		The top and bottom plots are for small and large matrices, respectively.
		The noise level is $\| e \|_2 / \| \bar{b} \|_2 = 0.003$.}
	\label{fig:histories}
\end{figure}

\begin{figure}
	\centering
	\includegraphics[width=0.9\textwidth]{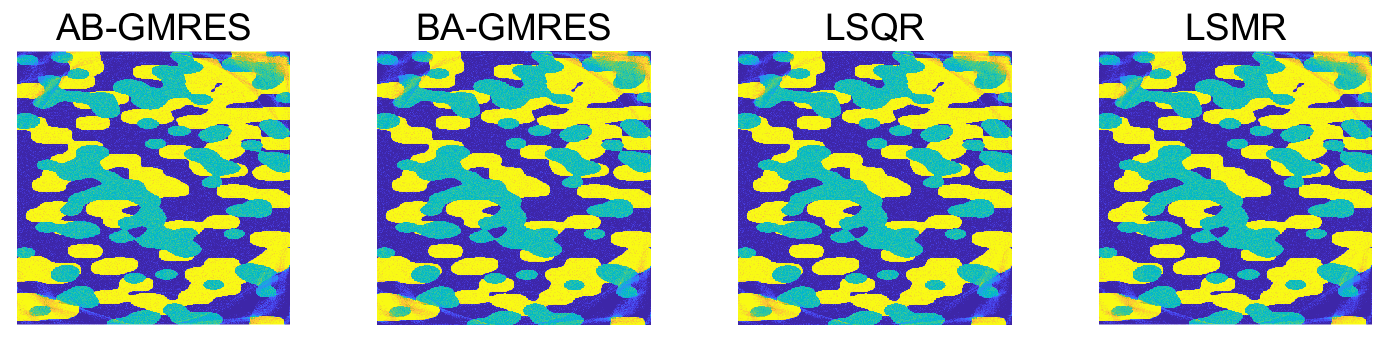}
	\caption{Reconstructions for the pair $(A,B) = (\As,\AlT)$ at the
		point of semi-convergence, i.e., the minima of the curves in the top
		right plot of Figure \ref{fig:histories}.
		Pixel values below zero and above one are set to zero and one, respectively.
		The slight artifacts near the corners are always present, and they are due
		to the scanning geometry with few X-ray penetrating these regions.}
	\label{fig:SLlarge}
\end{figure}

\begin{table}
	\caption{\label{table:recerr}
		For each algorithm and each $(A,B)$ pair, the numbers to the left and right
		in each block column are results for large and small test matrices, respectively.
		We list the reconstruction error, the number of iterations (in parenthesis),
		and on the second line the required storage.
		Note that the results for LSQR and LSMR are independent of $B$.
		See Table~\ref{table:diffAA} for the norm-wise differences between the matrices.}
	\medskip
	\scriptsize
	\centering
	\renewcommand{\arraystretch}{1.5}\setlength{\tabcolsep}{3pt}
	\begin{tabular}{|c|rr|rr|cc|cc|} \hline
		Pair & \multicolumn{2}{c|}{AB-GMRES} & \multicolumn{2}{c|}{BA-GMRES} &
		\multicolumn{2}{c|}{LSQR} & \multicolumn{2}{c|}{LSMR} \\ \hline
		$(\As,\AlT)$ & 0.1047 (34) & 0.1203 (36) & 0.1039 (38) & 0.1199 (41)
		& 0.0954 (57) & 0.0996 (62) & 0.0946 (68) & 0.0990 (77) \\
		& 8568000 & 829440 & 6703200 & 671744 & & & & \\ 
		$(\As,\AiT)$ & 0.0985 (41) & 0.1059 (42) & 0.0978 (47) & 0.1056 (50)
		& 0.0954 (53) & 0.0996 (61) & 0.0946 (63) & 0.0990 (77) \\
		& 10332000 & 967680 & 8290800 & 819200 & & & & \\ \hline
		$(\Al,\AsT)$ & 0.0878 (61) & 0.0926 (77) & 0.0871 (74) & 0.0921 (85)
		& 0.0896 (52) & 0.0879 (64) & 0.0875 (63) & 0.0795 (80) \\
		& 15372000 & 1774080 & 13053600 & 1392640 & & & & \\ 
		$(\Al,\AiT)$ & 0.0899 (50) & 0.0927 (60) & 0.0894 (60) & 0.0922 (69)
		& 0.0896 (52) & 0.0870 (64) & 0.0875 (63) & 0.0795 (80) \\
		& 12600000 & 1382400 & 10584000 & 1130496 & & & & \\ \hline
		$(\Ai,\AlT)$ & 0.1005 (39) & 0.1112 (42) & 0.0998 (44) & 0.1105 (47)
		& 0.0935 (53) & 0.0948 (61) & 0.0928 (63) & 0.0940 (77) \\
		& 9828000 & 967680 & 7761600 & 770048 & & & & \\ 
		$(\Ai,\AsT)$ & 0.0915 (55) & 0.0930 (68) & 0.0901 (67) & 0.0920 (84)
		& 0.0935 (53) & 0.0948 (61) & 0.0928 (63) & 0.0940 (77) \\
		& 13860000 & 1566720 & 11818800 & 1376256 & & & & \\ \hline
	\end{tabular}
\end{table}

To illustrate the semi-convergence we use both AB-GMRES and BA-GMRES
to solve systems with small and large test matrices and with
the noise level $\| e \|_2 / \| \bar{b} \|_2 = 0.003$.
We tried all six pairs with $B\neq A$, and for comparison we also used LSQR and LSMR
(which correspond to the case $B = A^\trans$).
To check our implementations, we verified numerically that AB- and BA-GMRES
with $B = A^\trans$ give the same results LSQR and LSMR (we do not show there results here).
Figure \ref{fig:histories} shows the error histories, i.e., the
relative reconstruction error
$\|\bar{x}-x_k\|_2/\|\bar{x}\|_2$ versus the number of iterations for
the cases $(A,B) = (\As,\AiT)$ and $(A,B) = (\As,\AlT)$.
Figure \ref{fig:SLlarge} shows the reconstructions for the case
$(A,B) = (\As,\AlT)$.
These results are representative for all six pairs of matrices.

When we refer to the ``reconstruction error,'' we mean the relative error
at the point of semi-convergence (i.e., the minimum of the error histories)
indicated by the markers in the plots.
The reconstruction errors for all $(A,B)$ pairs are shown in Table~\ref{table:recerr}.
For each of the six $(A,B)$ pairs we observe the following:

\begin{itemize}
	\item
	We obtain almost the same reconstruction error for AB-GMRES and BA-GMRES,
	cf.\ block columns 2 and 3.
	\item
	We obtain almost the same reconstruction error for LSQR and LSMR,
	cf.\ block  columns 4 and 5.
	\item
	When $A=\As$ then LSQR and LSMR give slightly smaller reconstruction errors
	than AB- and BA-GMRES, cf.\ the top block row as well as Figure \ref{fig:histories}.
	\item
	For the other two $A$ matrices, all four methods give almost the same
	reconstruction errors, cf. the middle and bottom block rows.
	\item
	The pairs $(A,B) = (\Al,\AsT)$ and $(A,B) = (\Al,\AiT)$
	give marginally more accurate reconstructions than the other pairs,
	cf.\ the middle block row.
	\item
	Often, AB-GMRES uses just slightly fewer iterations than BA-GMRES -- and occasionally
	is uses significantly less iterations.
	\item
	LSQR always uses less iterations than LSMR.
\end{itemize}
We also carried out experiments with under-determined problems (which are
not documented here). The conclusions remain the same.

Table \ref{table:recerr} also lists the required amount of storage for the orthogonal
$q_k$-vectors which are of length $m$ and $n$ for AB-GMRES and BA-GMRES, respectively.
For the overdetermined systems used here with $m \approx 1.4 n$,
in spite of BA-GMRES consistently using more iterations than AB-GMRES,
BA-GMRES requires less storage due to the shorter $q_k$-vectors.
For underdetermined systems (not reported here), AB-GMRES has the
advantage of less iterations and shorter $q_k$-vectors,
cf.\ \cite[p.~2408]{HaYI10}.

We emphasize that the choice of $A$ and $B$ is dictated by the available software,
and therefore one may not always have a choice of the implementation used in $A$ and $B$.
Moreover, the above results are for matrices used in the ASTRA software
(which allows easy access to the matrices);
other packages may use different discretization methods.
We also stress that in these experiments we perform inverse crime, meaning that
the noise-free data is generated as $\bar{b} = A\,\bar{x}$; hence the data is
different for the three choices of $A$ meaning that we do not solve
precisely the same problem for each choice of~$A$.
Therefore, the above results provide important insight about the influence
of $B\neq A^\trans$,
but they do not determine what is the best choice of $A$ and $B$ for
a CT problem with real data and no inverse crime.

\subsection{Stopping Rules}
\label{sec:sr}

Here we demonstrate the use of two stopping rules that seek to terminate
the iterations at the point of semi-convergence.
\begin{itemize}
	\item
	The \textbf{discrepancy principle (DP)} \cite{DP} terminates the iterations as
	soon as the residual norm is smaller than the noise level:
	\begin{equation}
		k_{\mathrm{DP}} = \hbox{the smallest $k$ for which
			$\| b - A\, x_k \|_2 \leq \tau\, \| e \|_2$ \ .}
	\end{equation}
	Here, $\tau\geq 1$ is a ``safety factor'' that can be used when we have
	only a rough estimate of $\| e \|_2$.
	We use $\tau = 1$ and the exact value of~$\| e \|_2$.
	\item
	The \textbf{NCP criterion} uses the normalized cumulative
	periodogram to perform a spectral analysis
	of the residual vector $b - A\, x_k$, in order to
	identify when the residual is as close to being white noise as possible,
	which indicates that all available information has been extracted
	from the noisy data.
	See \cite[Section~2.3.3]{AIRtoolsII} and \cite[SectionII.D]{HaJR21} for details;
	MATLAB code is available from us.
\end{itemize}

\begin{figure}
	\centering
	AB-GMRES \qquad\qquad $(\As,\AlT)$ \qquad\qquad BA-GMRES \\
	\includegraphics[width=0.7\textwidth]{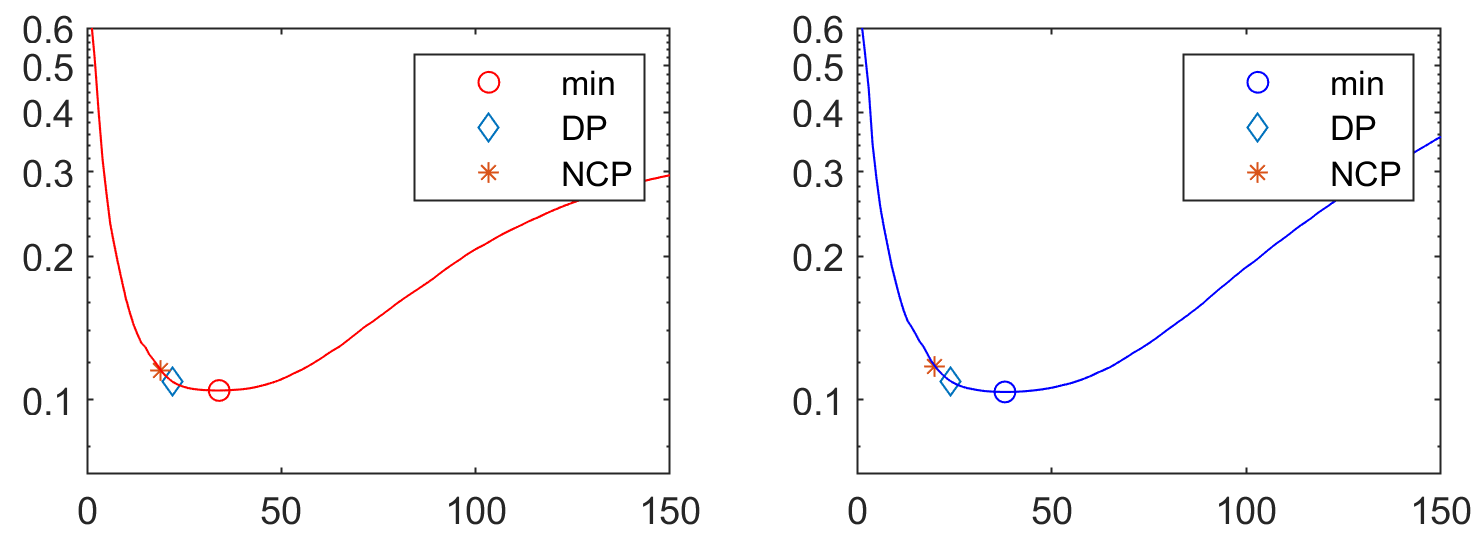} \\[2mm]
	AB-GMRES \qquad\qquad $(\As,\AiT)$ \qquad\qquad BA-GMRES \\
	\includegraphics[width=0.7\textwidth]{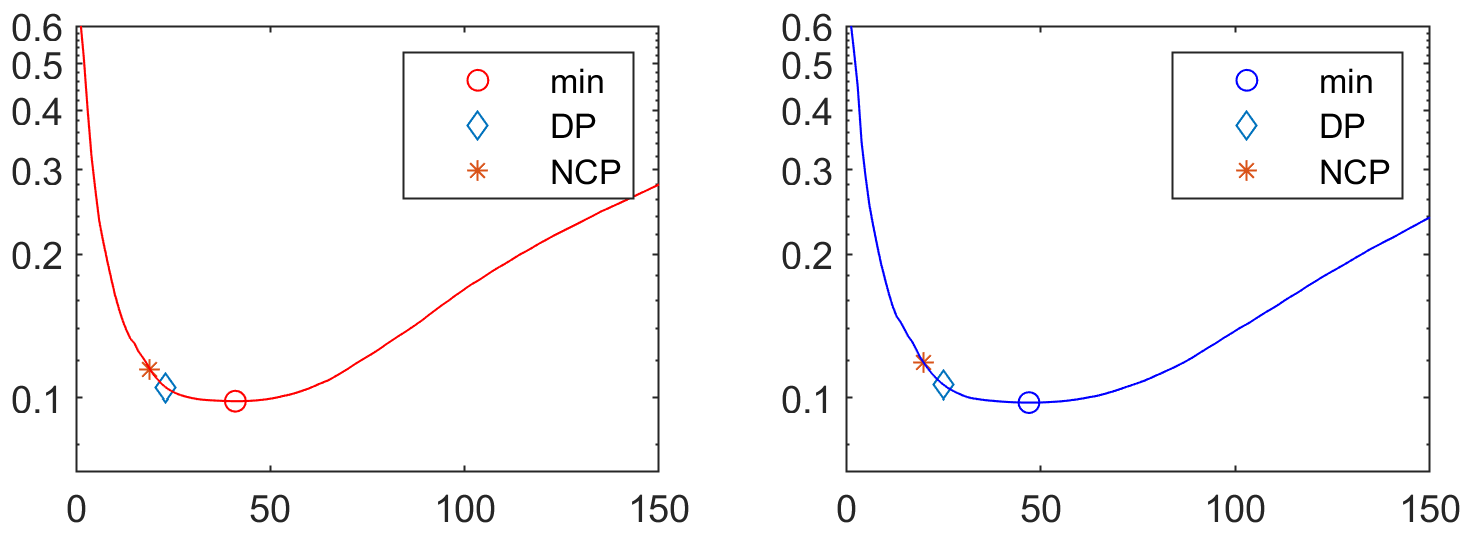} \\[2mm]
	LSQR \qquad\qquad\qquad $\As$ \qquad\qquad\qquad LSMR \\
	\includegraphics[width=0.7\textwidth]{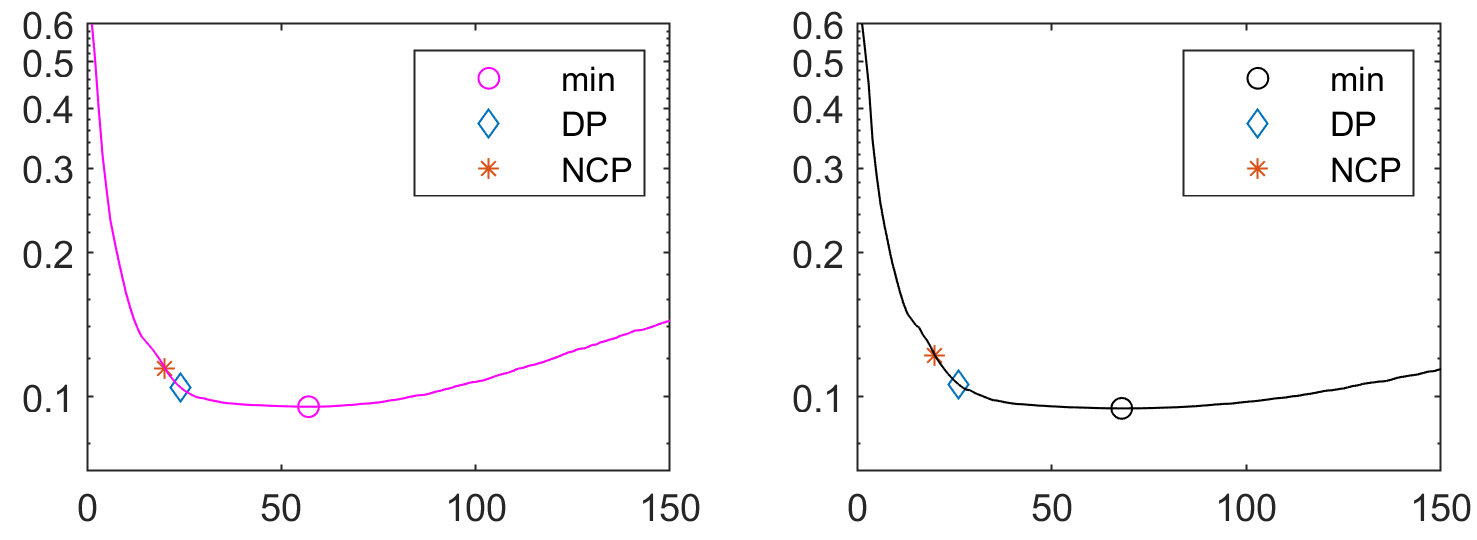}
	\caption{We show the error histories $\|\bar{x}-x_k\|_2/\|\bar{x}\|_2$
		versus the number of iterations.
		The markers indicate the minimum as well as the error at the
		number of iterations $k_{\mathrm{DP}}$ and $k_{\mathrm{NCP}}$
		found by the discrepancy principle and the NCP criterion, respectively.
		We show AB- and BA-GMRES results for the cases $(A,B) = (\As,\AiT)$ and
		$(A,B) = (\As,\AlT)$, as well as LSQR and LSMR results for $\As$,
		similarly to the case of large matices in
		Figure~\ref{fig:histories} and with the same
		noise level $\| e \|_2 / \| \bar{b} \|_2 = 0.003$.}
	\label{fig:stoprules}
\end{figure}

We apply these stopping rules to the same pairs of large matrices
$(A,B) = (\As,\AiT)$ and $(A,B) = (\As,\AlT)$ as in Figure~\ref{fig:histories},
with the same noise level $\| e \|_2 / \| \bar{b} \|_2 = 0.003$,
and the results are shown in Figure~\ref{fig:stoprules}.
We obtain similar results for the other $(A,B)$ pairs and hence they
are not shown here.
We make the following observations:
\begin{itemize}
	\item
	Both DP and NCP stop the iterations before the minimum is reached.
	This is better than stopping too late, in which case we would include
	undesired noise in the solution.
	\item
	Both DP and NCP stop the iterations when the error history starts to level off;
	the minimum of the error history is quite flat so this is acceptable.
	\item
	For unknown reasons, NCP always stops the iterations a bit earlier than DP.
\end{itemize}
We conclude that both stopping rules work well for this problem.
The DP stopping rule requires a good estimate of the noise level;
if this is not available (as is typical in CT problems)
then the performance of NCP is only slightly inferior to DP.

\subsection{SVD Analysis of Semi-Convergence}

\begin{figure}
	\centering
	\includegraphics[width=0.37\textwidth]{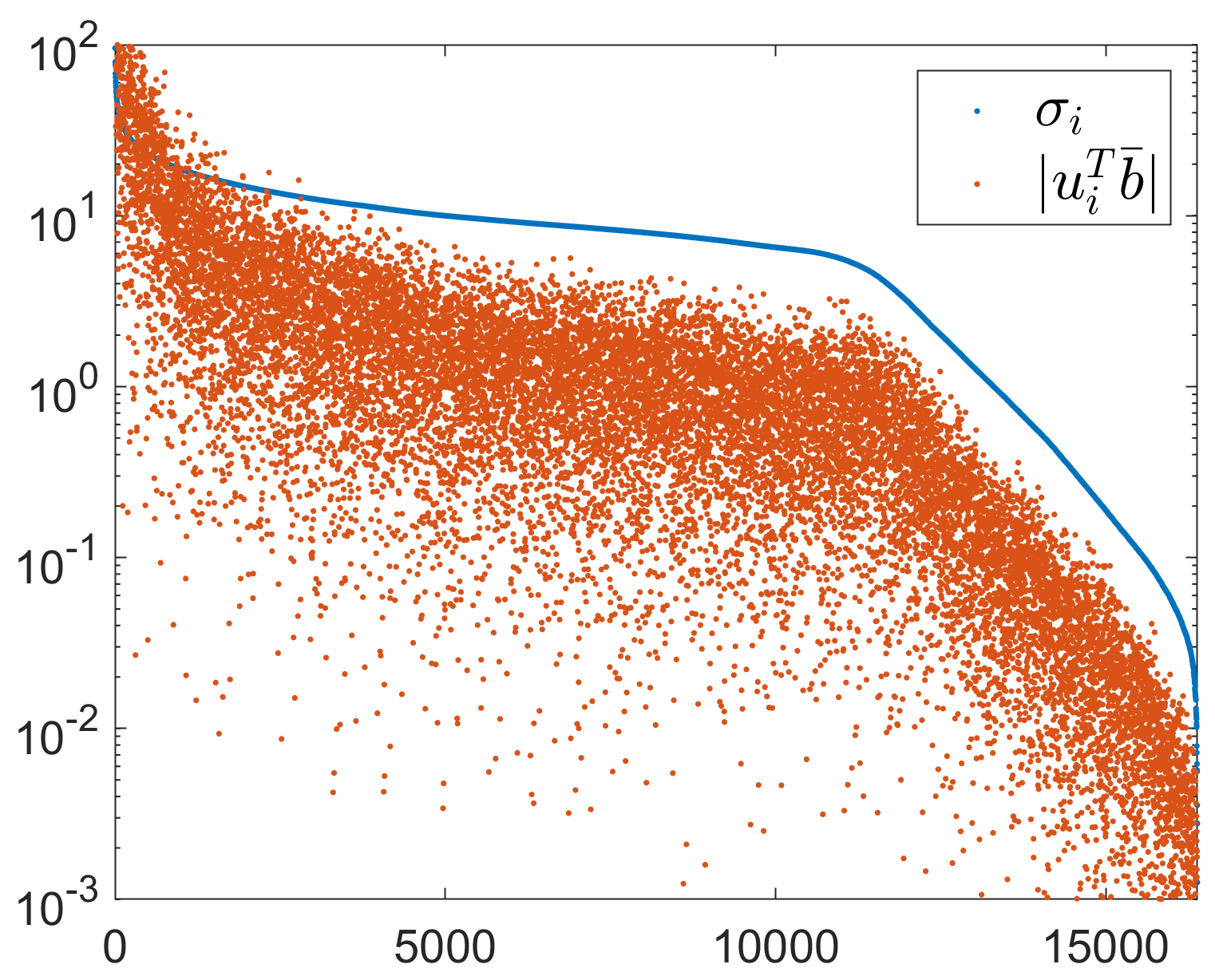} \hspace{2mm}
	\includegraphics[width=0.37\textwidth]{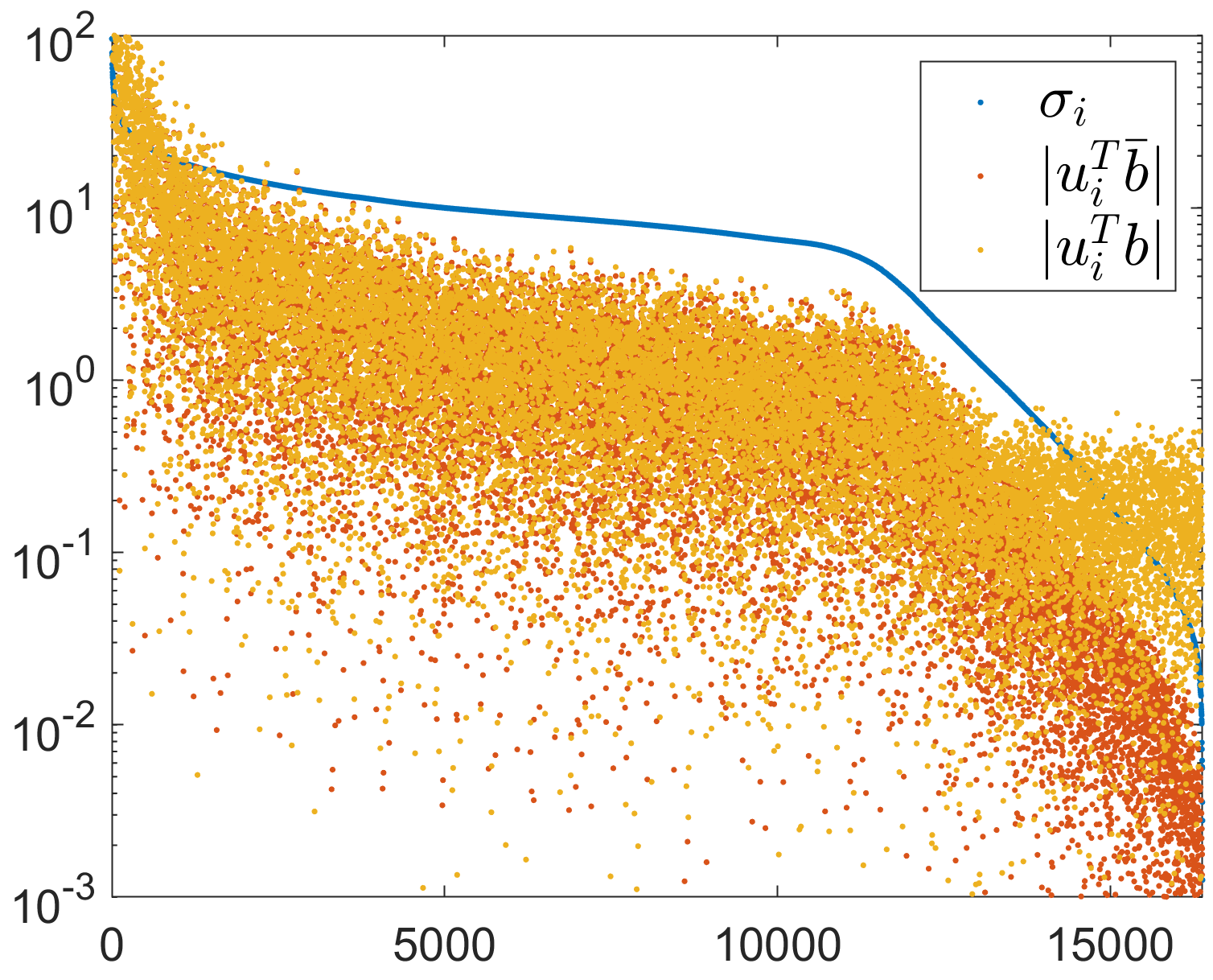}
	\caption{Plots of the singular values $\sigma_i$ (blue dots) and the right-hand side's
		SVD coefficients $|u_i^\trans\bar{b}|$ for the exact data (red dots)
		and $|u_i^\trans b|$ for noisy data (brown dots)
		with noise level $\| e \|_2 / \| \bar{b} \|_2 = 0.003$.
		The left plot shows only $\sigma_i$ and $|u_i^\trans\bar{b}|$ while the right
		plot shows all three quantities.}
	\label{fig:picard}
\end{figure}

\begin{figure}
	\centering
	\includegraphics[width=0.47\textwidth]{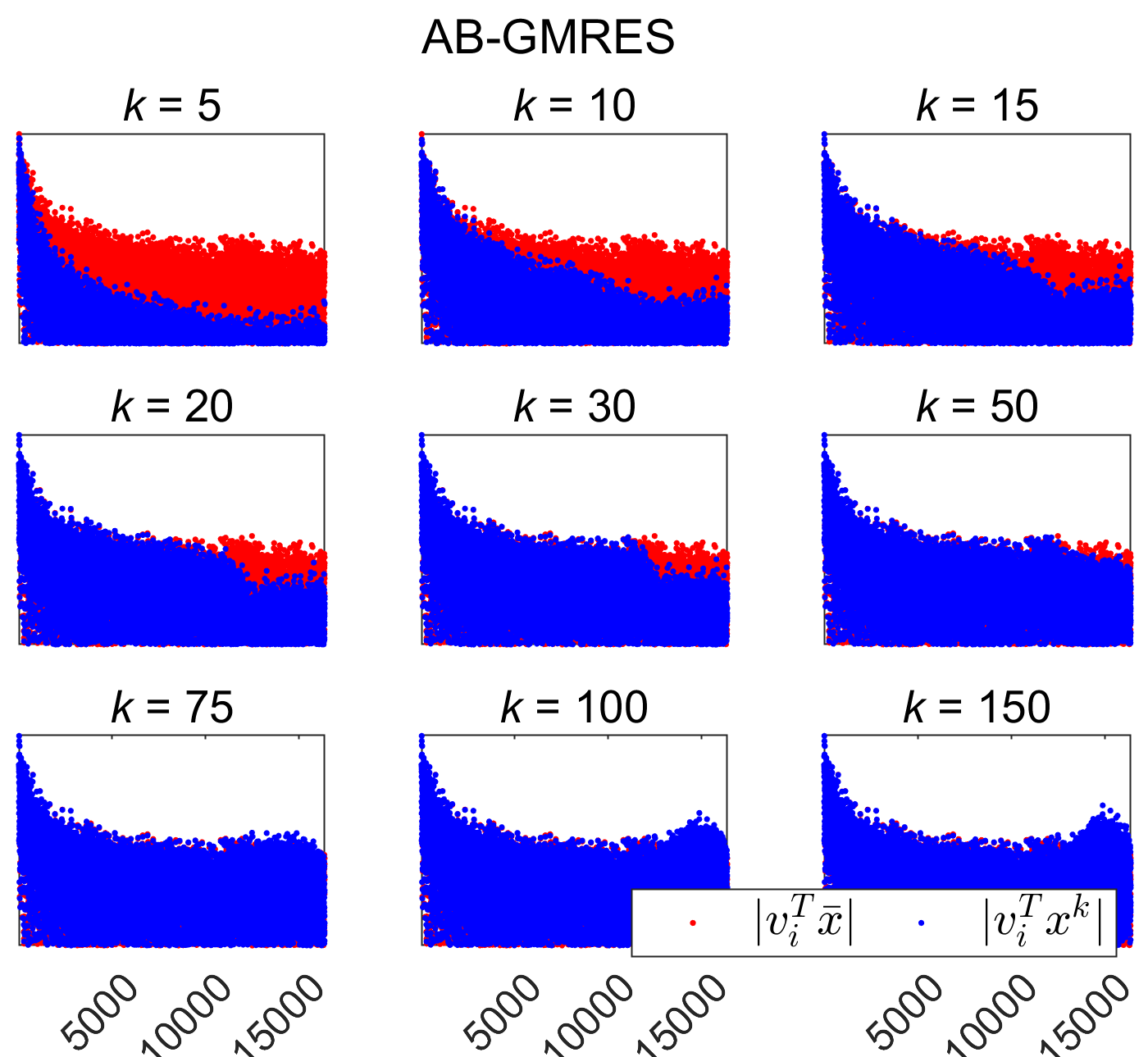} \qquad
	\includegraphics[width=0.47\textwidth]{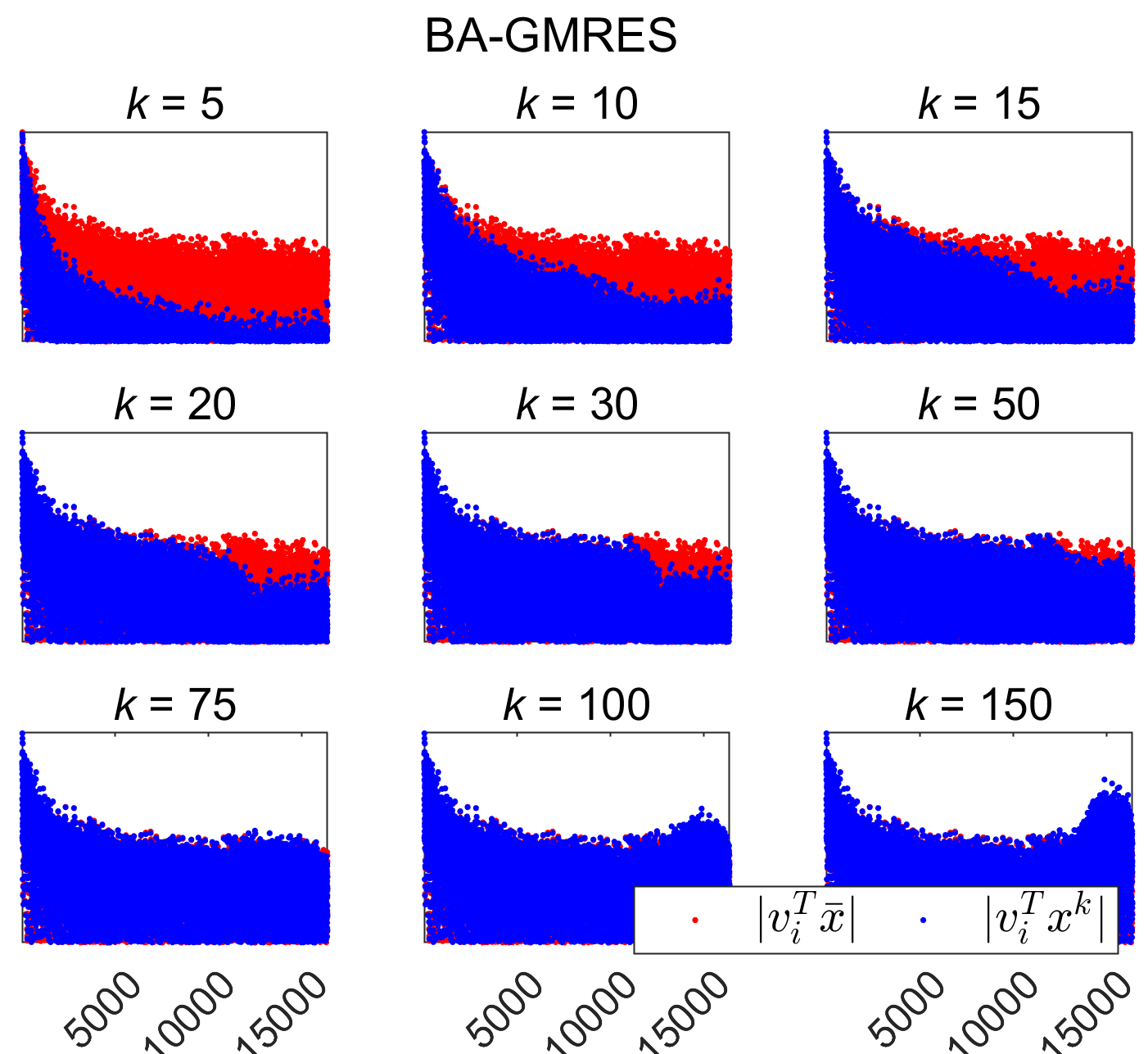} \\[3mm]
	\includegraphics[width=0.47\textwidth]{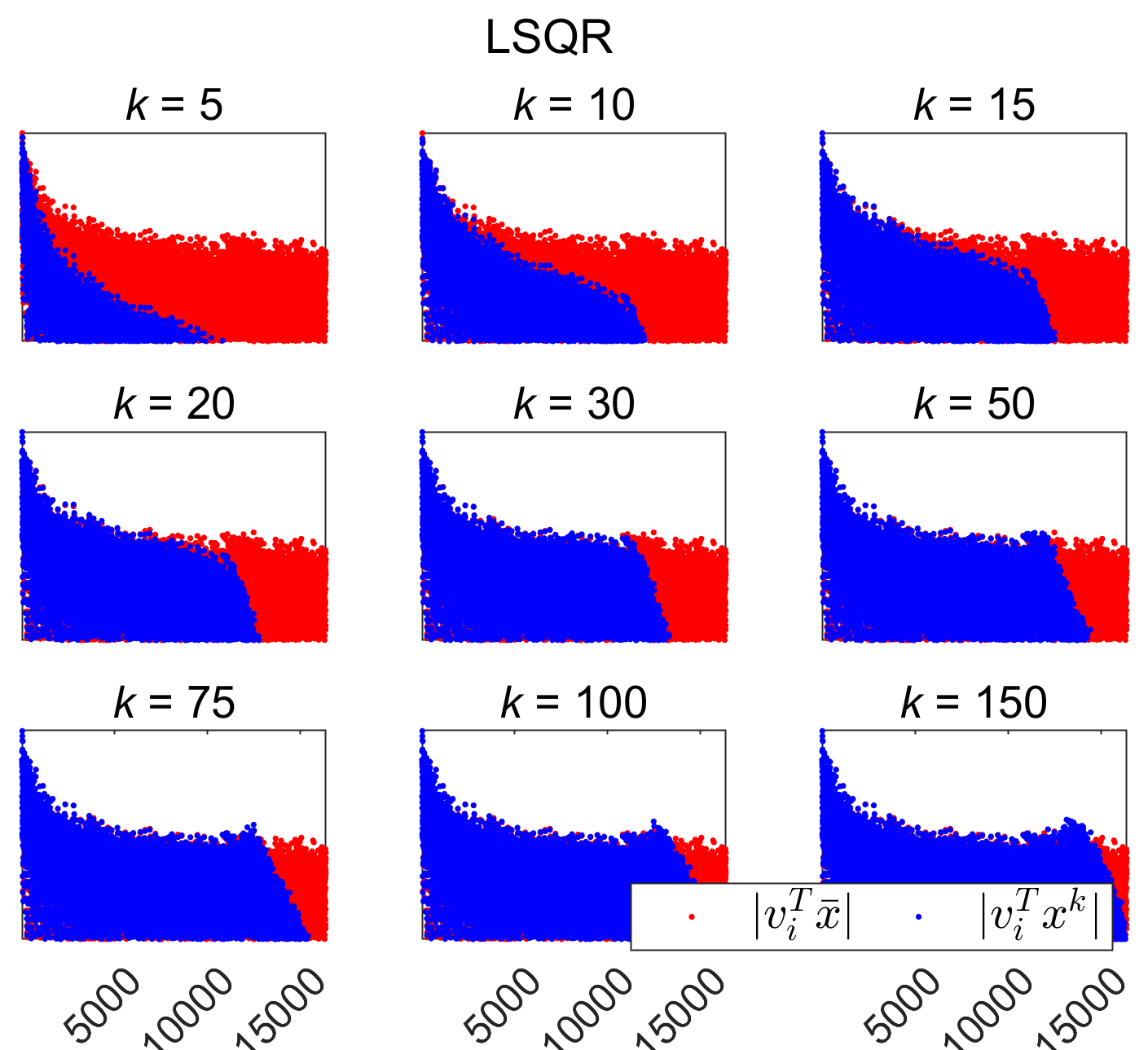} \qquad
	\includegraphics[width=0.47\textwidth]{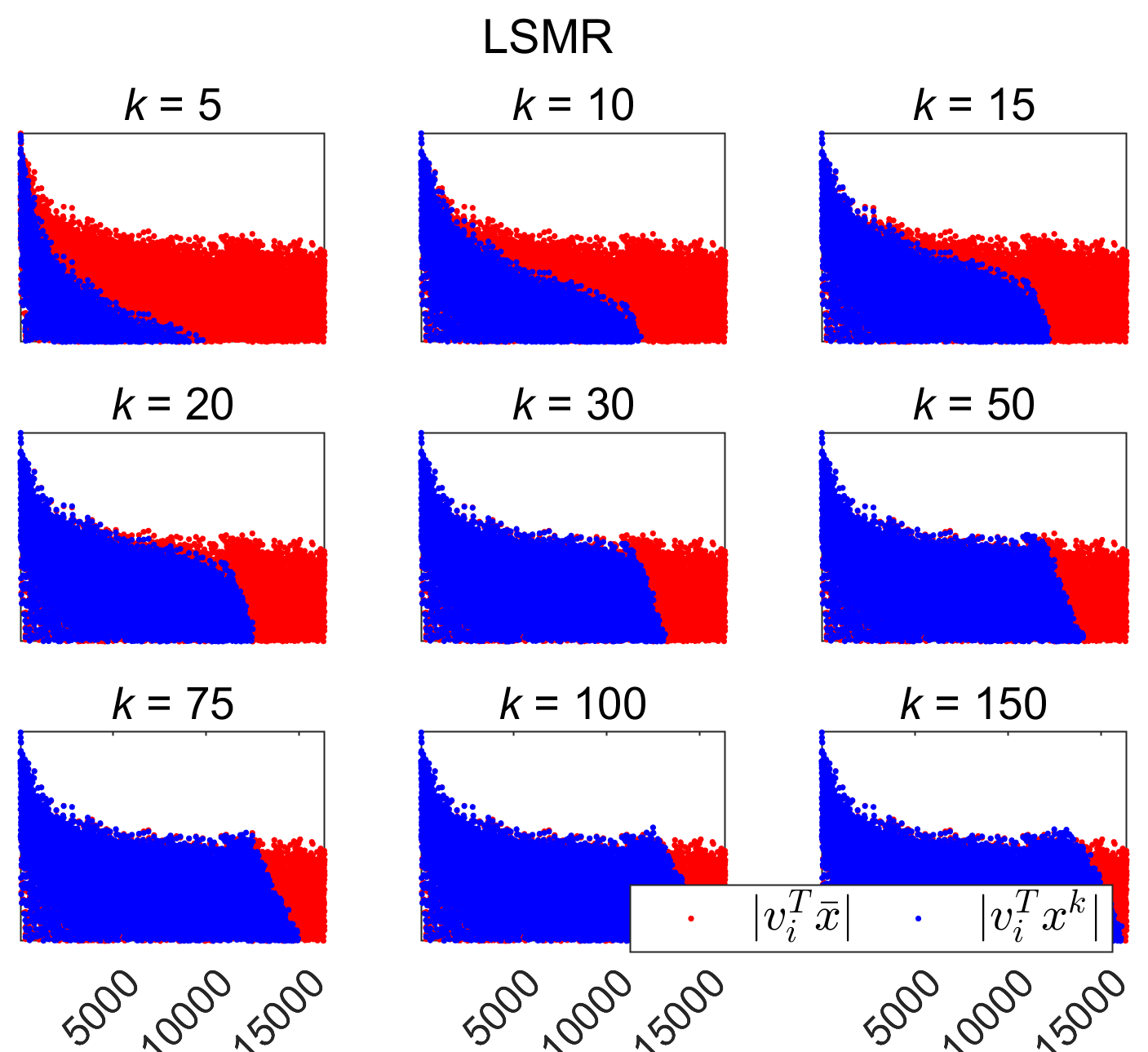}
	\caption{SVD analysis of the iterates $x_k$ for
		selected values of the iteration number~$k$, for the low noise level
		$\| e \|_2 / \| \bar{b} \|_2 = 0.003$.
		Red dots:\ SVD coefficients $|v_i^\trans \bar{x}|$ for the exact solution.
		Blue dots:\ SVD coefficients $|v_i^\trans x_k|$.}
	\label{fig:svd}
\end{figure}

More detailed insight into the semi-convergence can be obtained by
means of the SVD of $A$, and we present results for the small test matrices
$(A,B) = (\As,\AlT)$
(for the large matrices we were not able to compute enough SVD components
with MATLAB's \texttt{svds}).
Figure \ref{fig:picard} shows the singular values together with the
right-hand side's SVD coefficients
$|u_i^\trans\bar{b}|$ for the exact data and $|u_i^\trans b|$ for noisy data
with noise level $\| e \|_2 / \| \bar{b} \|_2 = 0.003$.
The behavior of the singular values is typical for discretizations of
CT problems -- all the large ones decay like those of the underlying
Radon transform while the smaller ones decay faster due to discretization effects.
We see that the exact data satisfy the DPC, i.e., they decay at least as
fast as the singular values.
We also clearly see the ``noise floors'' in the right plots of $|u_i^\trans b|$
around 0.5.

Figure \ref{fig:svd} shows, for selected values of~$k$,
the SVD coefficients $|v_i^\trans\bar{x}|$ (red) for the exact solution
together with the SVD coefficients $|v_i^\trans x_k|$ (blue) for
selected iterates $x_k$ with
noisy data with noise level $\| e \|_2 / \| \bar{b} \|_2 = 0.003$.
Again we use the small test matrices.
For all four methods, as $k$ increases we capture an increasing amount of
SVD components, and at $k=30$ we have already computed
good approximations of the first $10\,000$ exact SVD components.
As we perform further iterations we start to capture unwanted SVD components
associated with small singular values; these components are
influenced by noise causing the blue dots to form a ``bump'' that
lies distinctly above the red dots.
This is particularly pronounced for AB- and BA-GMRES after $k=100$ iterations.
These ``bumps'' are obviously not present for noise-free data (not shown here).

In Figure \ref{fig:svd} we also observe a distinctly different behavior of
LSQR and LSMR versus AB- and BA-GMRES\@.  As discussed, e.g.,
in \cite[Section 6.3.2]{DIPbook}, LSQR is a spectral filtering method
that produces filtered SVD solution conforming with \eqref{eq:fSs},
and the same is true for LSMR which uses the same Krylov subspace.
For the smaller singular values the filter factors behave as
$\phi_i^{(k)} = O(\sigma_i^2)$ meaning that they decay fast to zero,
causing the blue dots
to ``drop off'' very fast -- for example, at $k=30$ iterations
the LSQR and LSMR iterates contain practically no SVD coefficients
with index $i > 12\,000$.
As we perform more iterations, eventually a small ``bump'' starts to appear
for large indices $i>10\,000$;
it is clearly visible at $k=150$ but it is much less pronounced than
for AB- and BA-GMRES.

AB- and BA-GMRES, on the other hand, are not spectral filtering methods
of the form \eqref{eq:fSs}, due to the underlying Krylov subspace
$\mathcal{K}_k(\BA, \Bb)$,
and SVD coefficients for all $i = 1,\ldots,n$ are present in all iterations
(some larger and some smaller).
This causes the noise to enter the iterates faster, and hence the reconstruction
error (at the point of semi-convergence) tends to be larger for AB- and BA-GMRES.

\begin{figure}
	\centering
	\includegraphics[width=0.47\textwidth]{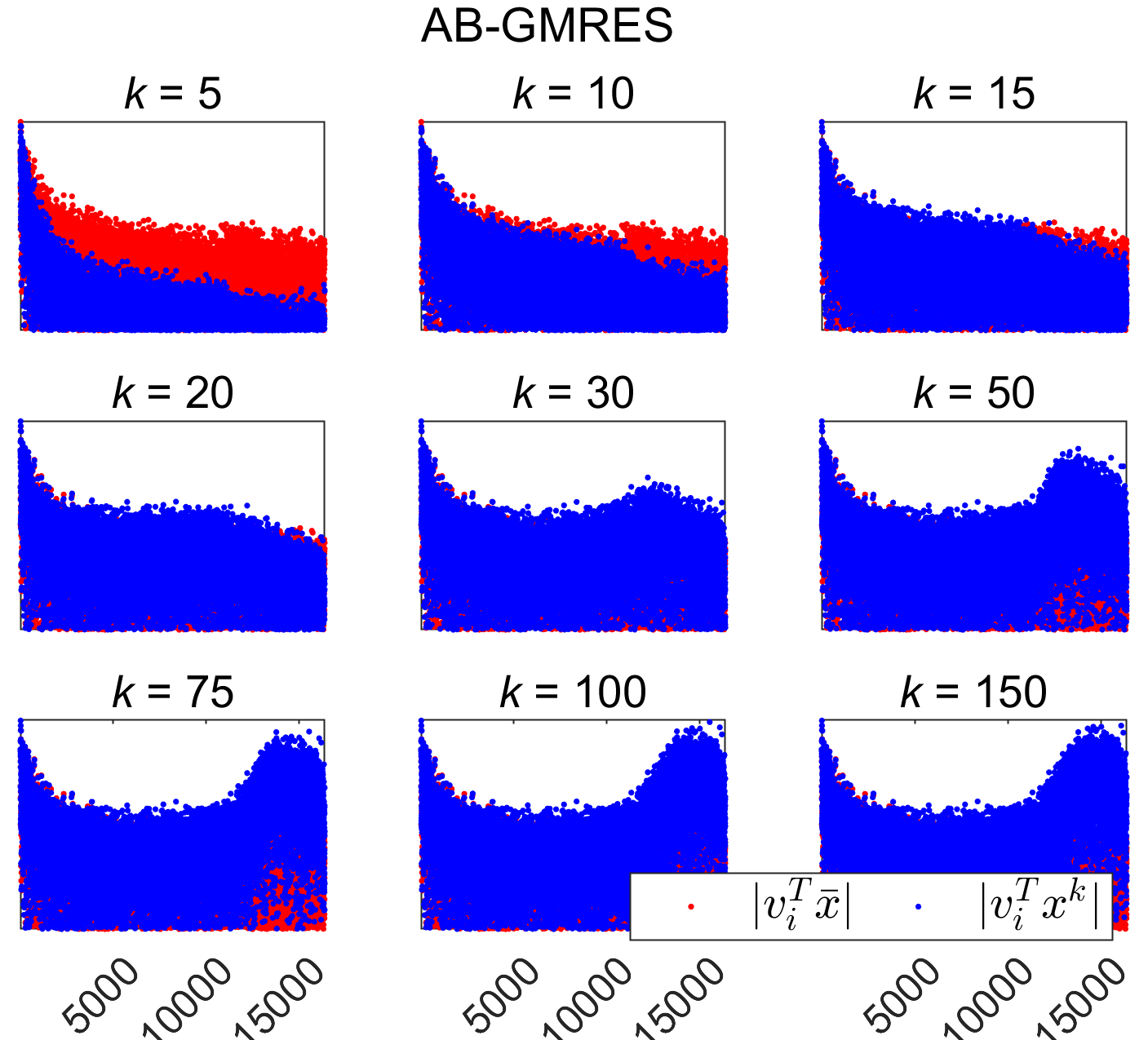} \qquad
	\includegraphics[width=0.47\textwidth]{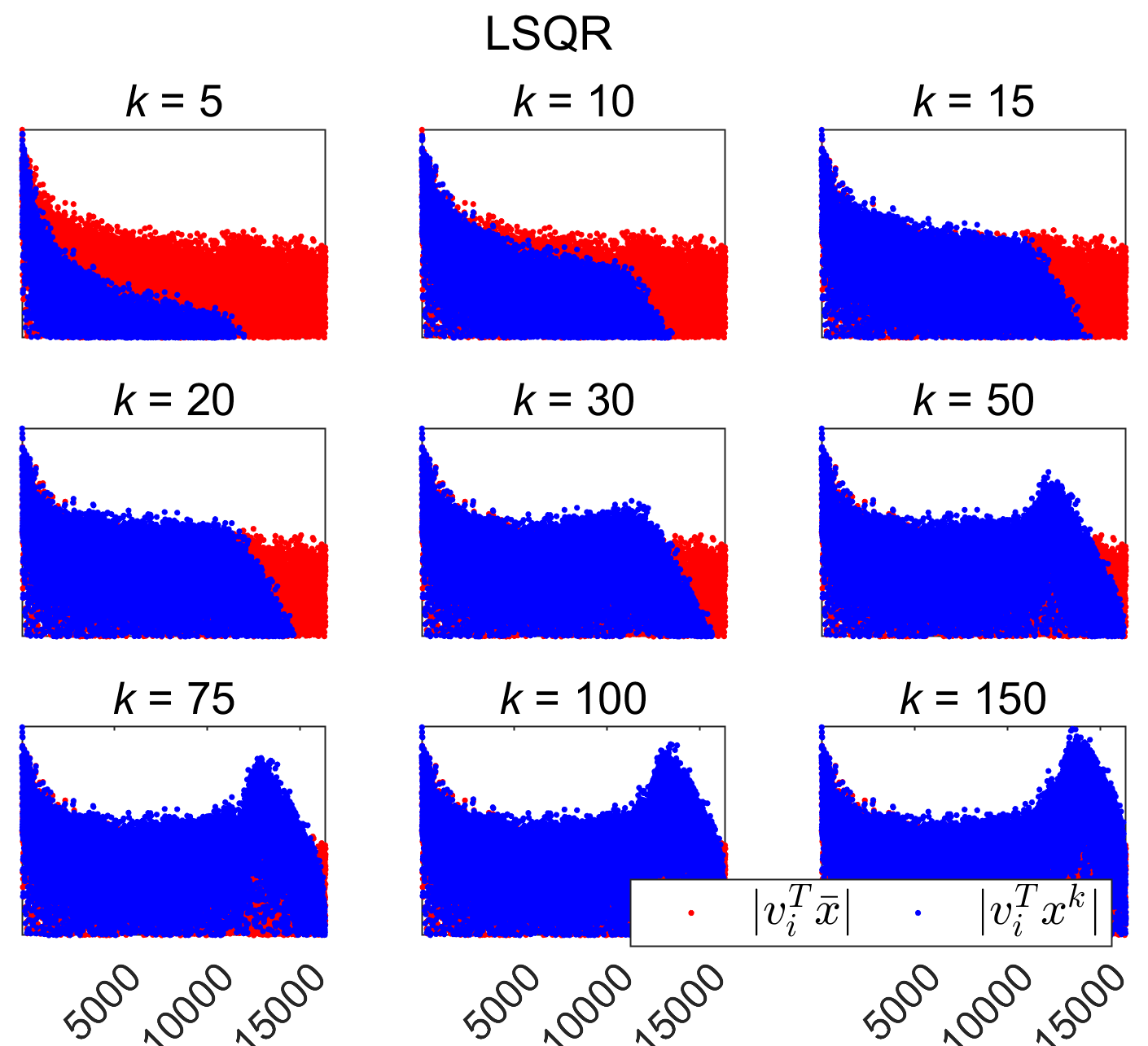}
	\caption{SVD analysis similar to Figure \ref{fig:svd} for the larger noise level
		$\| e \|_2 / \| \bar{b} \|_2 = 0.03$.}
	\label{fig:svdagain}
\end{figure}

Figure \ref{fig:svdagain} shows the SVD coefficients for AB-GMRES and LSQR,
similar to~Figure \ref{fig:svd}, for the larger noise level
$\| e \|_2 / \| \bar{b} \|_2 = 0.03$.
The overall behavior is similar to that in the previous figure, except that
the undesired ``bump'' now appears after only $k=30$ iterations.

\subsection{Varying the Back Projector's Unmatchedness}
\label{sec:unmatchedness}

We conclude with numerical experiments where we study the influence of the
back projector's \emph{unmatchedness}, as measured by
$\| B - A^\trans \|\fro / \| A \|\fro$.
So far we have considered matrices $A$ and $B$ generated by the ASTRA
software package, which are representative examples of the matrices
that we encounter with CT reconstruction software.
The disadvantage is that we cannot control ourselves how
much $B$ deviates from~$A^\trans$.
For this reason we introduce a new set of back projection matrices $B_\tau$
defined by neglecting elements of $A^\trans$ according to a threshold $\tau$.
Specifically, if $(\cdot)_{ij}$ denotes matrix elements then we define
\begin{equation}
	\label{eq:Btau}
	(B_\tau)_{ij} = \left\{ \begin{array}{ll}
		(A)_{ji} & \hbox{if} \quad (A)_{ji} \geq \tau\,\max_{ij}(A)_{ij} \\[1mm]
		0 & \hbox{else.} \end{array} \right.
\end{equation}
We do not need absolute values since all matrix elements are nonnegative.
When $\tau=0$ then $B=A^\trans$ is perfectly matched, and the larger the $\tau$
the more unmatched the~$B$.
We emphasize that these unmatched matrices do not represent actual
implementations of back projections in CT software packages.

\begin{figure}
	\centering
	\includegraphics[width=0.4\textwidth]{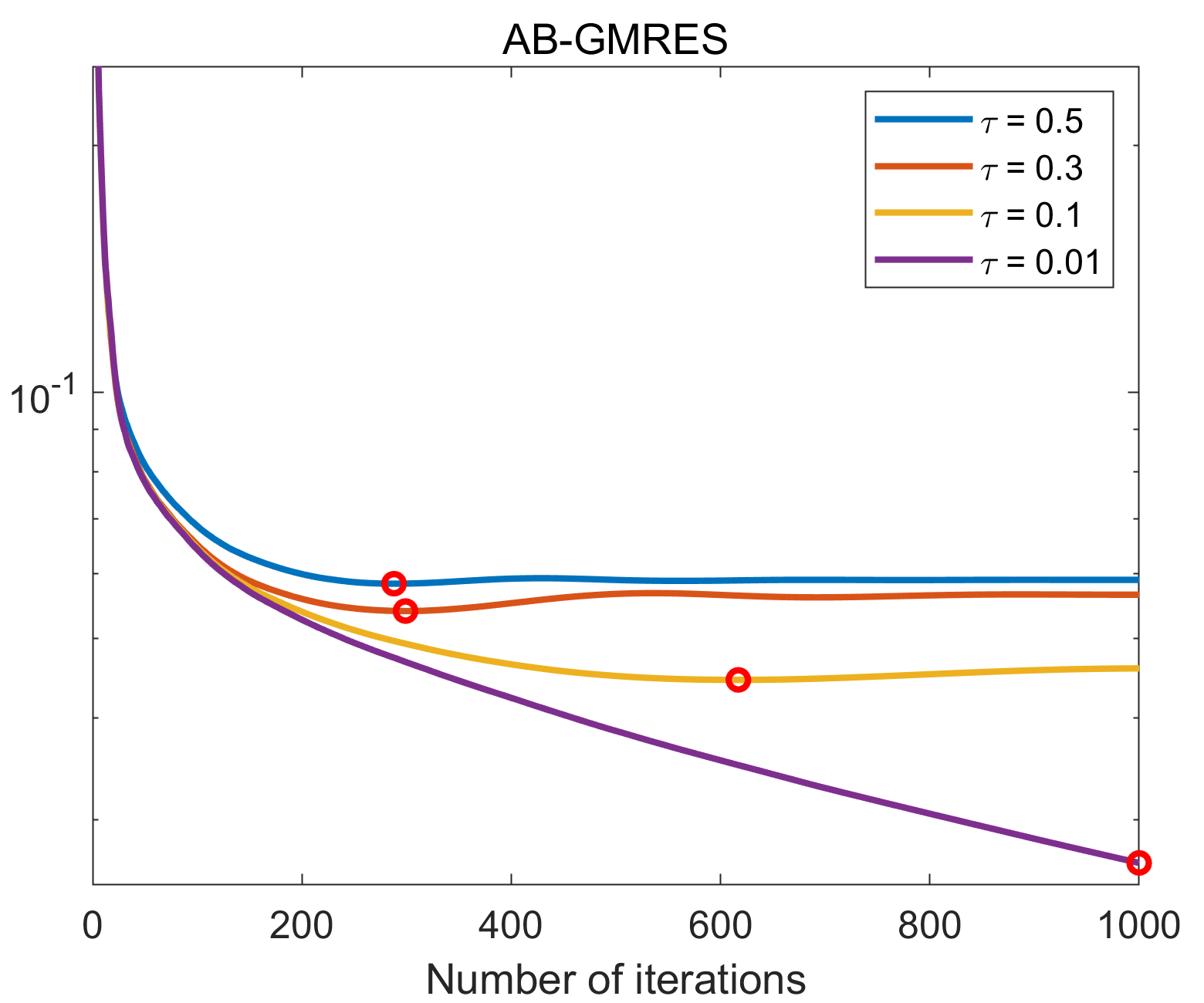} \quad
	\includegraphics[width=0.4\textwidth]{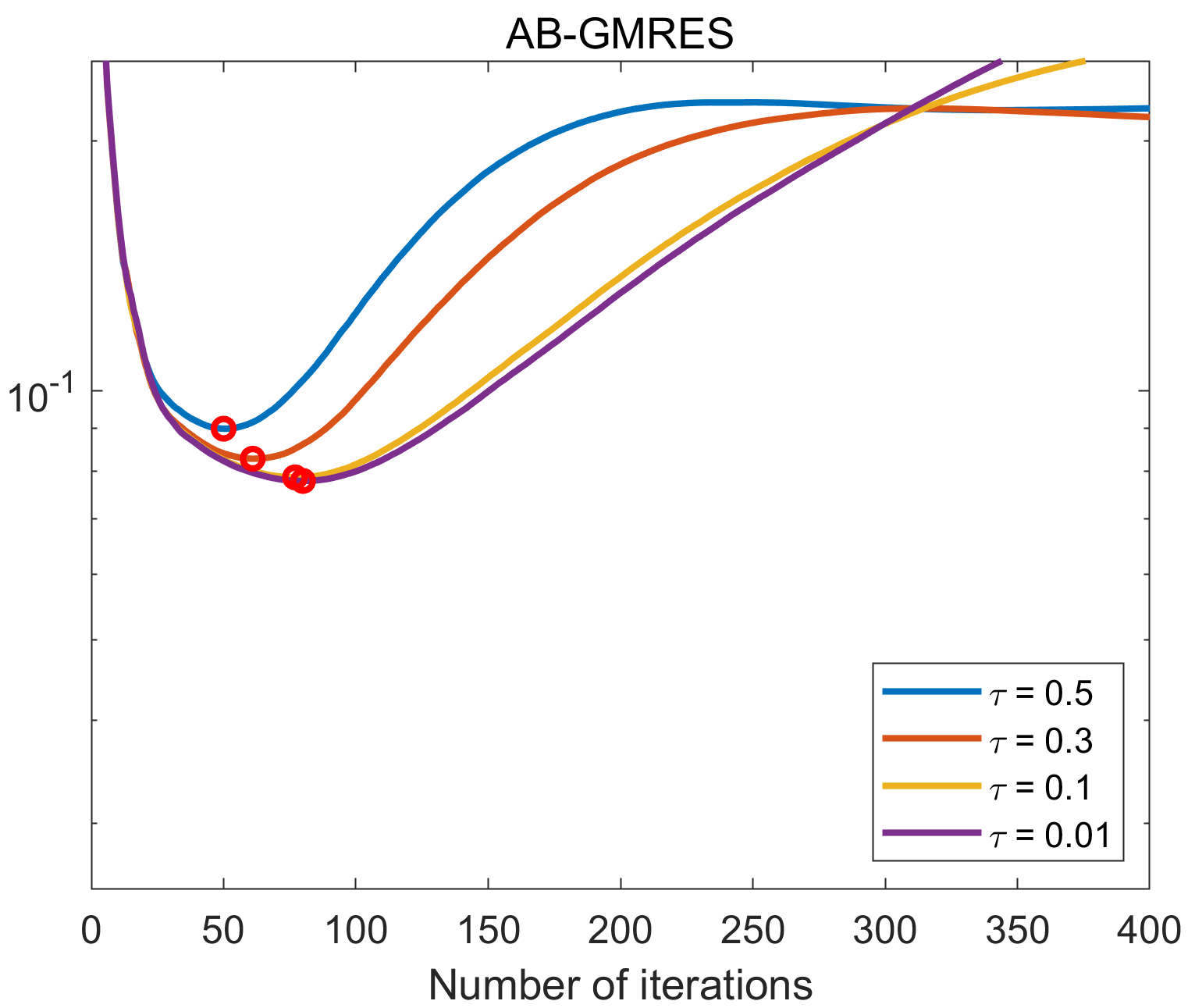} \\
	\includegraphics[width=0.4\textwidth]{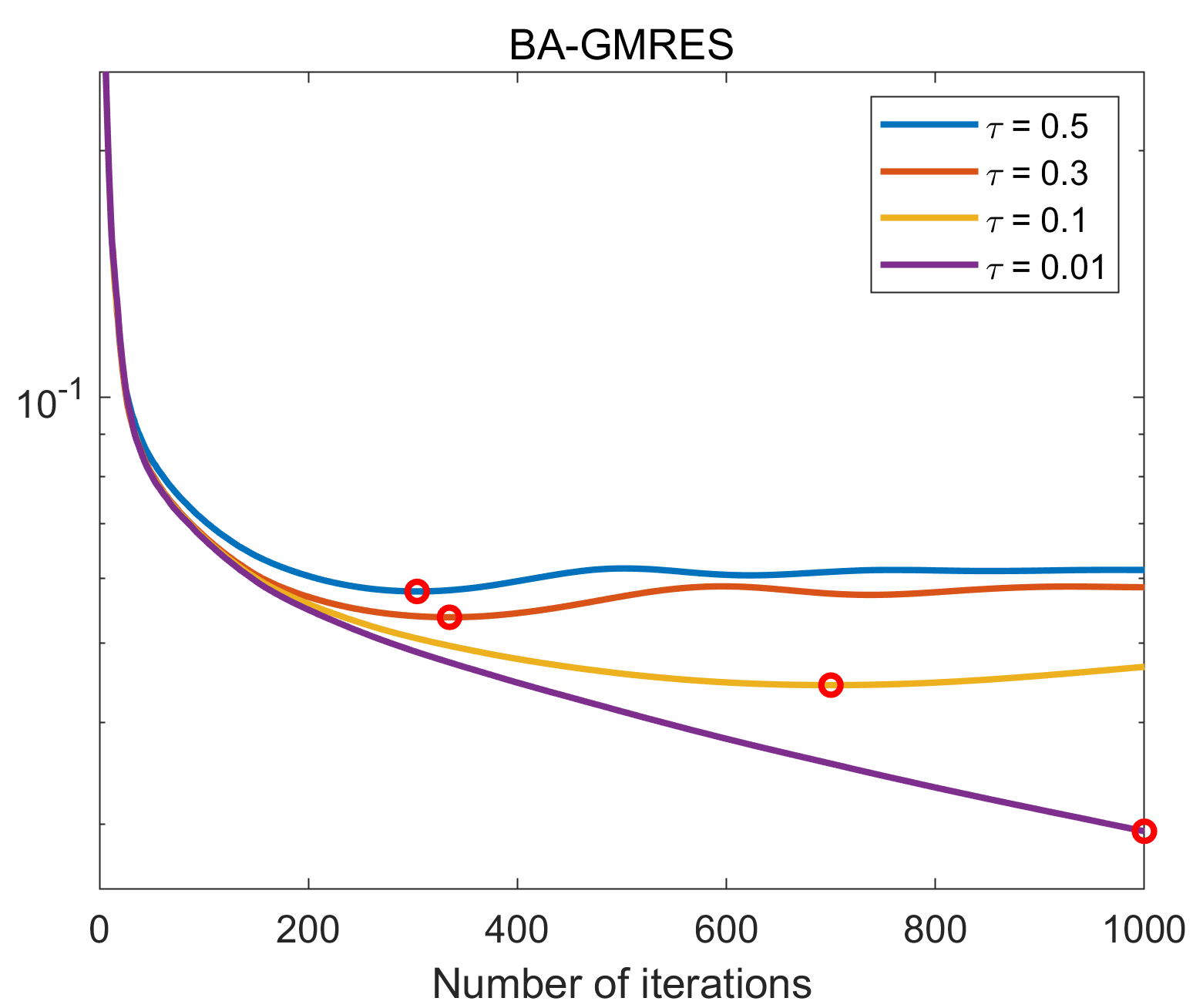} \quad
	\includegraphics[width=0.4\textwidth]{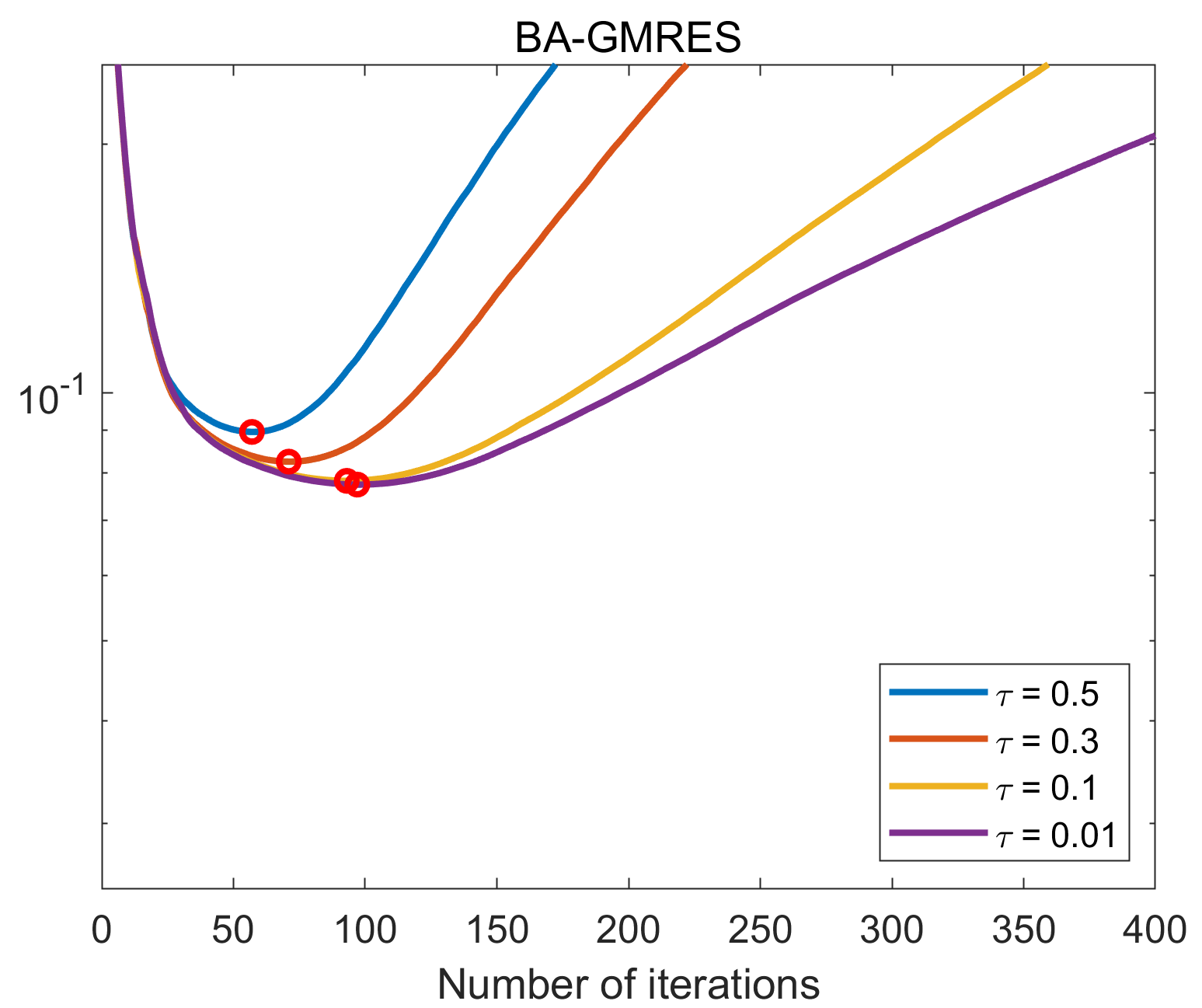}
	\caption{Error histories for four pairs $(\As,B_\tau)$ with
		four different values of $\tau$.
		The top and bottom plots are for the AB-GMRES and BA-GMRES methods, respectively.
		The left plots use a noise-free right-hand side $\bar{b}$ and the right
		plots are for noisy data $b=\bar{b}+e$ with $\| e \|_2 / \| \bar{b} \|_2 = 0.003$.
		The minima are marked with red circles.}
	\label{fig:InfluenceOfErorInBstrip}
\end{figure}

\begin{table}
	\caption{\label{tbl:unmatch}Corresponding values of $\tau$, the
		measure of unmatchedness.}
	\centering
	\renewcommand{\arraystretch}{1.3}
	\begin{tabular}{|c|ccccc|} \hline
		$\tau$ & 0.01 & 0.1 & 0.3 & 0.5  & comment \\ \hline
		\multirow{2}{*}{$\| B_\tau - \AsT \|\fro / \| \As \|\fro$}
		& 0.0021 & 0.0386 & 0.1640 & 0.3366 & small matrices \\
		& 0.0021 & 0.0387 & 0.1650 & 0.3397 & large matrices \\ \hline
	\end{tabular}
\end{table}

\begin{figure}
	\centering
	\includegraphics[width=0.5\textwidth]{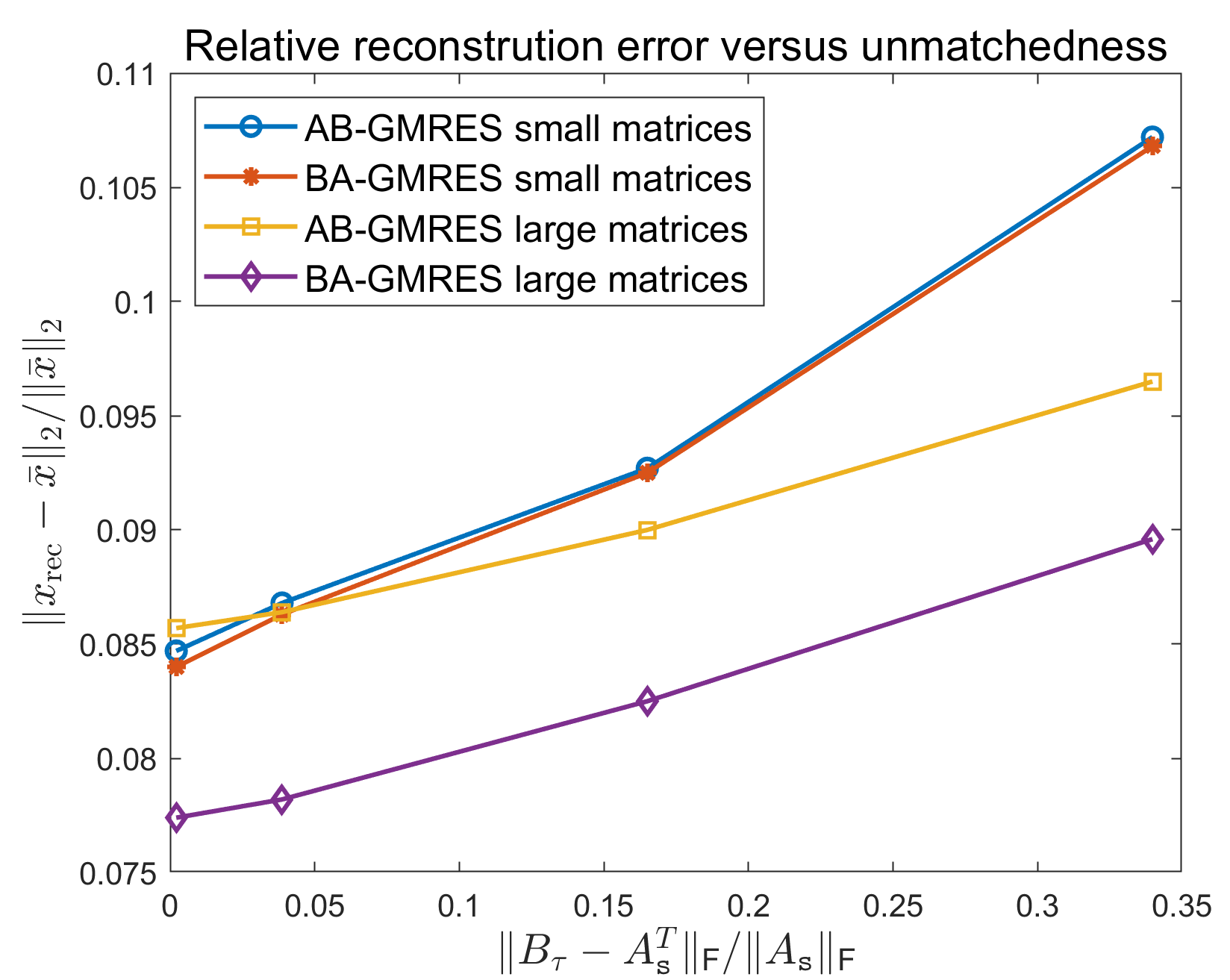}
	\caption{The relative reconstruction error versus the unmatchedness
		for the pair $(\As,B_\tau)$.}
	\label{fig:RREvsU}
\end{figure}

We focus on the large $252\,000 \times 176\,400$ test matrix $\As$
generated with the strip model.
Figure \ref{fig:InfluenceOfErorInBstrip} shows relative
error histories $\| \bar{x}-x_k \|_2 / \| \bar{x} \|_2$ for
both AB-GMRES and BA-GMRES, for difference choices of $\tau$,
see Table~\ref{tbl:unmatch}, and
with a noise-free as well as a noisy right-hand side with
$\| e \|_2 / \| \bar{b} \|_2 = 0.003$.
We make the following observations:
\begin{enumerate}
	\item
	Both methods converge to the solution $x_\tau^{\mathrm{UNE}}$ to the
	unmatched normal equations \eqref{eq:unmatchednormalequations}.
	When $B_\tau = \AsT$ and $e=0$ (no noise) then $x_\tau^{\mathrm{UNE}}$
	equals the ground truth $\bar{x}$,
	and they differ otherwise.
	The difference $x_\tau^{\mathrm{UNE}} - \bar{x}$ increases as $\tau$ increases
	and as $\| e \|_2$ increases.
	\item
	The reconstruction error
	consists of two components:\ the iteration error
	$x_\tau^{\mathrm{UNE}} - x_k$ and the error $\bar{x} - x_\tau^{\mathrm{UNE}}$
	caused by $B_\tau \neq A^\trans$.
	For small $\tau$ the first component dominates, while both components
	may contribute for large~$\tau$.
	\item
	For noise-free data the minimum error is entirely determined by the
	norm $\| B_\tau - \AsT \|\fro$.
	\item
	For noisy data, where we have semi-convergence,
	for all four values of~$\tau$ the minimum is mainly determined by the error $e$.
	\item
	For $\tau=0.5$ we have $\| B_\tau - \AsT \|\fro / \| \As \|\fro = 0.34$
	and the error histories
	resemble those for the pair $(\As,\AlT)$ in Figure~\ref{fig:histories}, for which
	$\| \As - \AlT \|\fro / \| \As \|\fro = 0.37$ (almost the same amount of unmatchedness).
	\item
	For $\tau=0.01$ we have that $B_\tau$ is almost a matched transpose, and
	the error histories resemble those for LSQR and LSMR in Figure~\ref{fig:histories}.
\end{enumerate}
These results confirm our theory, namely, that the behavior of AB-GMRES and
BA-GMRES resembles that of LSQR and LSMR, respectively, when $\tau \rightarrow 0$.
Moreover, for large $\tau$ the behavior of AB- and BA-GMRES with $B = B_\tau$
resembles that with the ASTRA matrices.

Figure \ref{fig:RREvsU} shows the relative reconstruction error
$\| x_{\mathrm{rec}} - \bar{x} \|_2 / \| \bar{x} \|_2$ versus the
unmatchedness, where $x_{\mathrm{rec}}$ is the iteration vector at the
minimum (the point of semi-convergence).
We observe a linear dependence between the two quantities.

\begin{figure}
	\centering
	----------- $\tau = 0.01$ ----------- \qquad
	----------- $\tau = 0.1$ ----------- \qquad  \\[3mm]
	\includegraphics[width=0.35\textwidth]{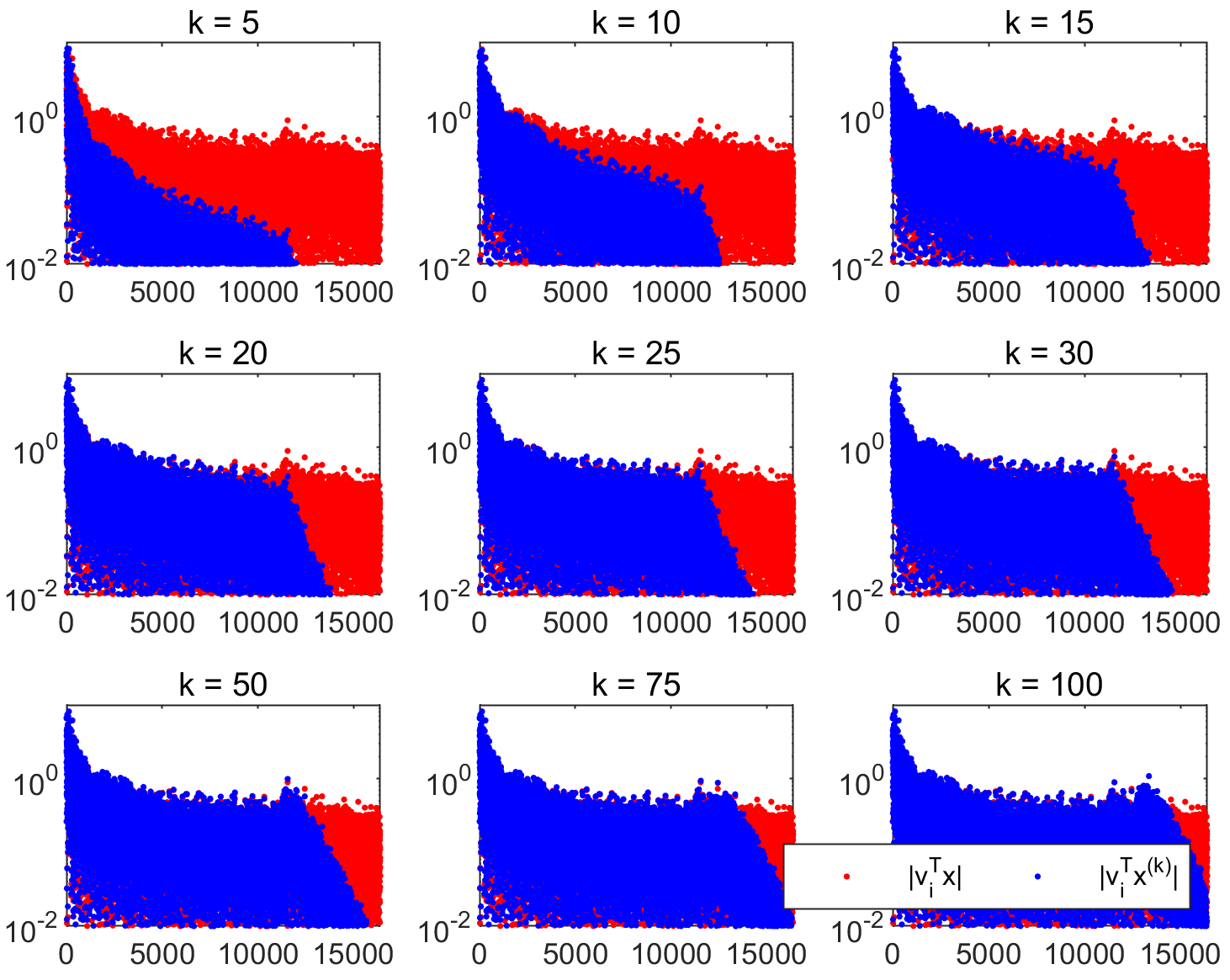}
	\includegraphics[width=0.35\textwidth]{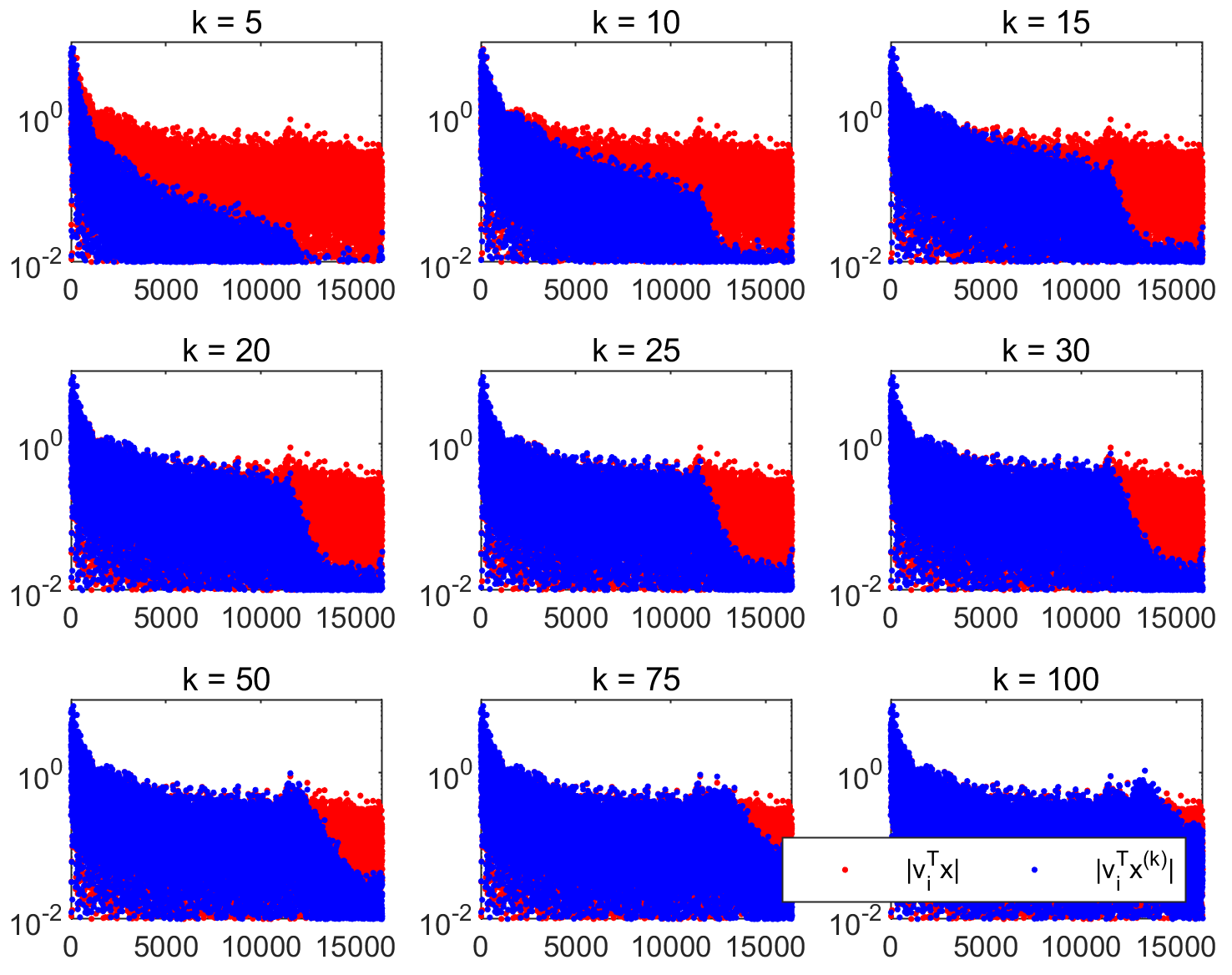} \\[3mm]
	\includegraphics[width=0.35\textwidth]{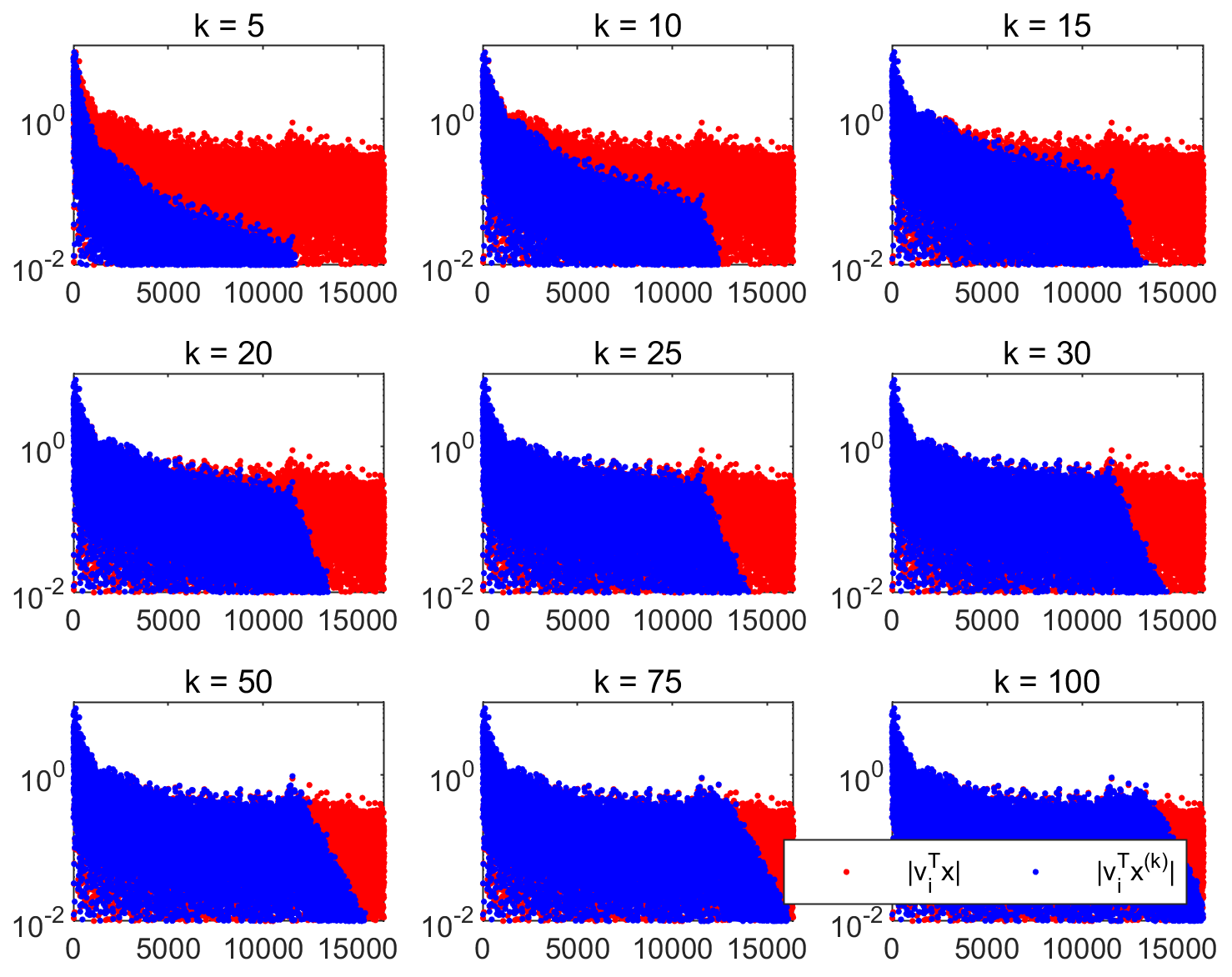}
	\includegraphics[width=0.35\textwidth]{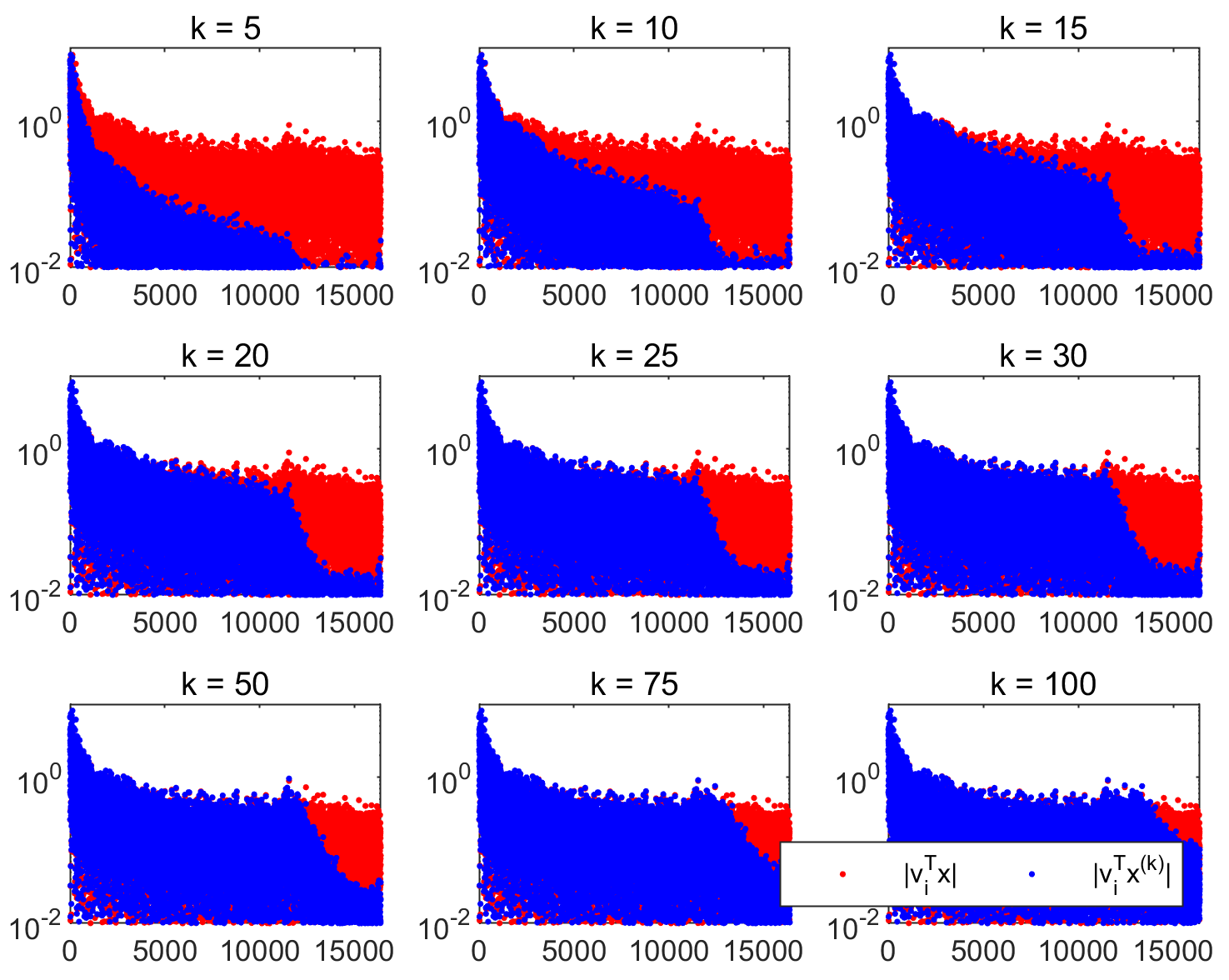} \\[3mm]
	--------- $\tau = 0.3$ --------- \qquad
	--------- $\tau = 0.5$ --------- \\[3mm]
	\includegraphics[width=0.35\textwidth]{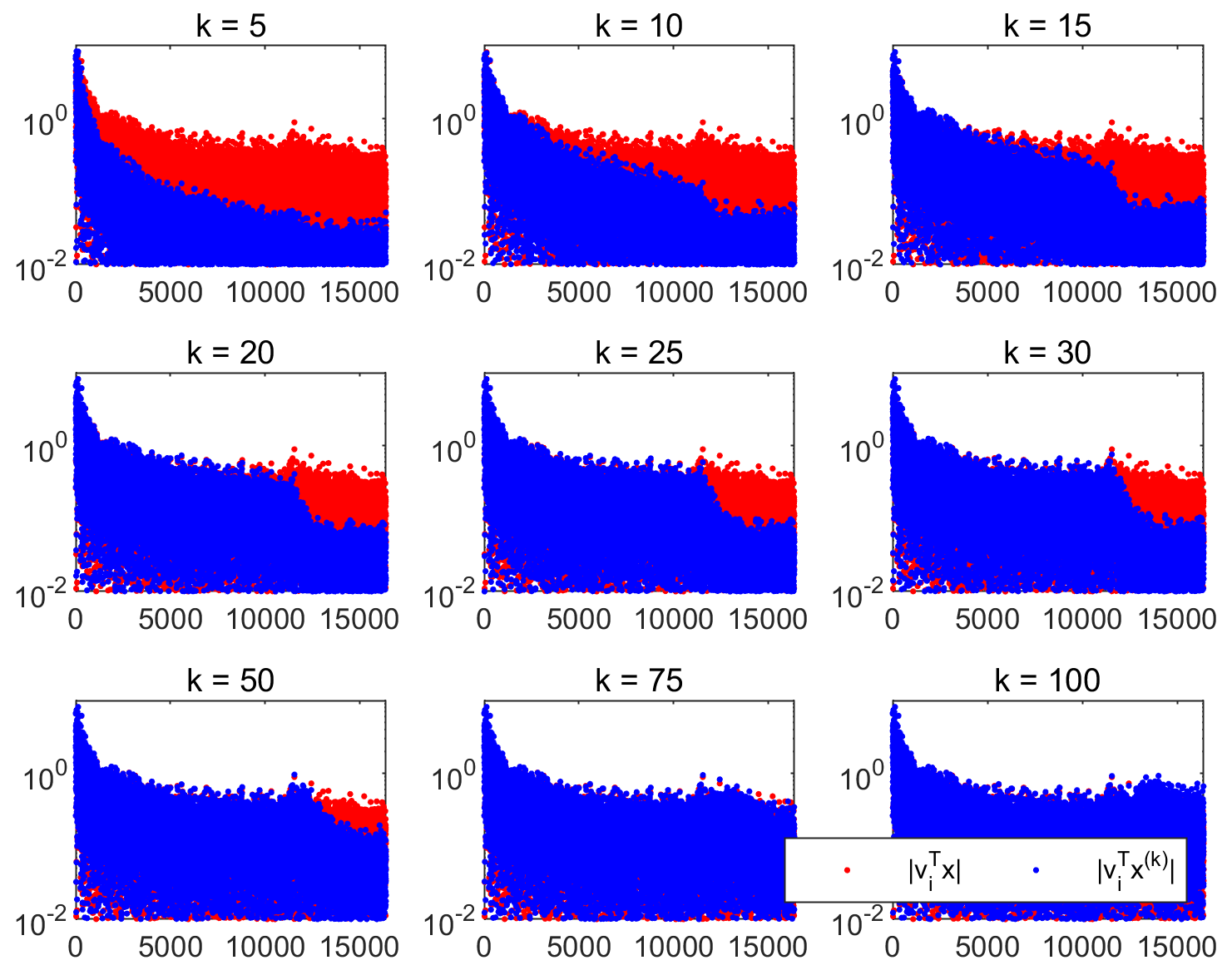}
	\includegraphics[width=0.35\textwidth]{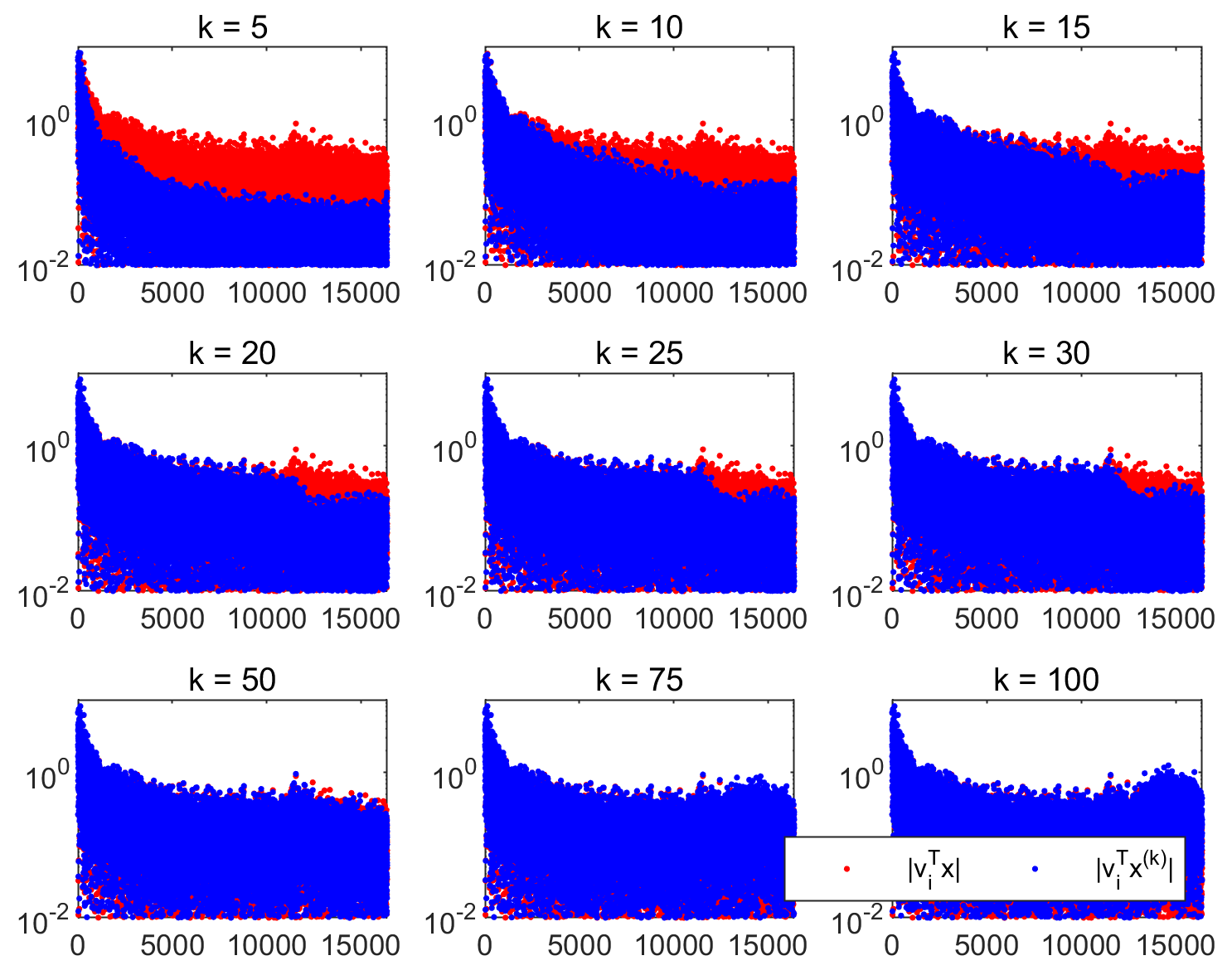} \\[3mm]
	\includegraphics[width=0.35\textwidth]{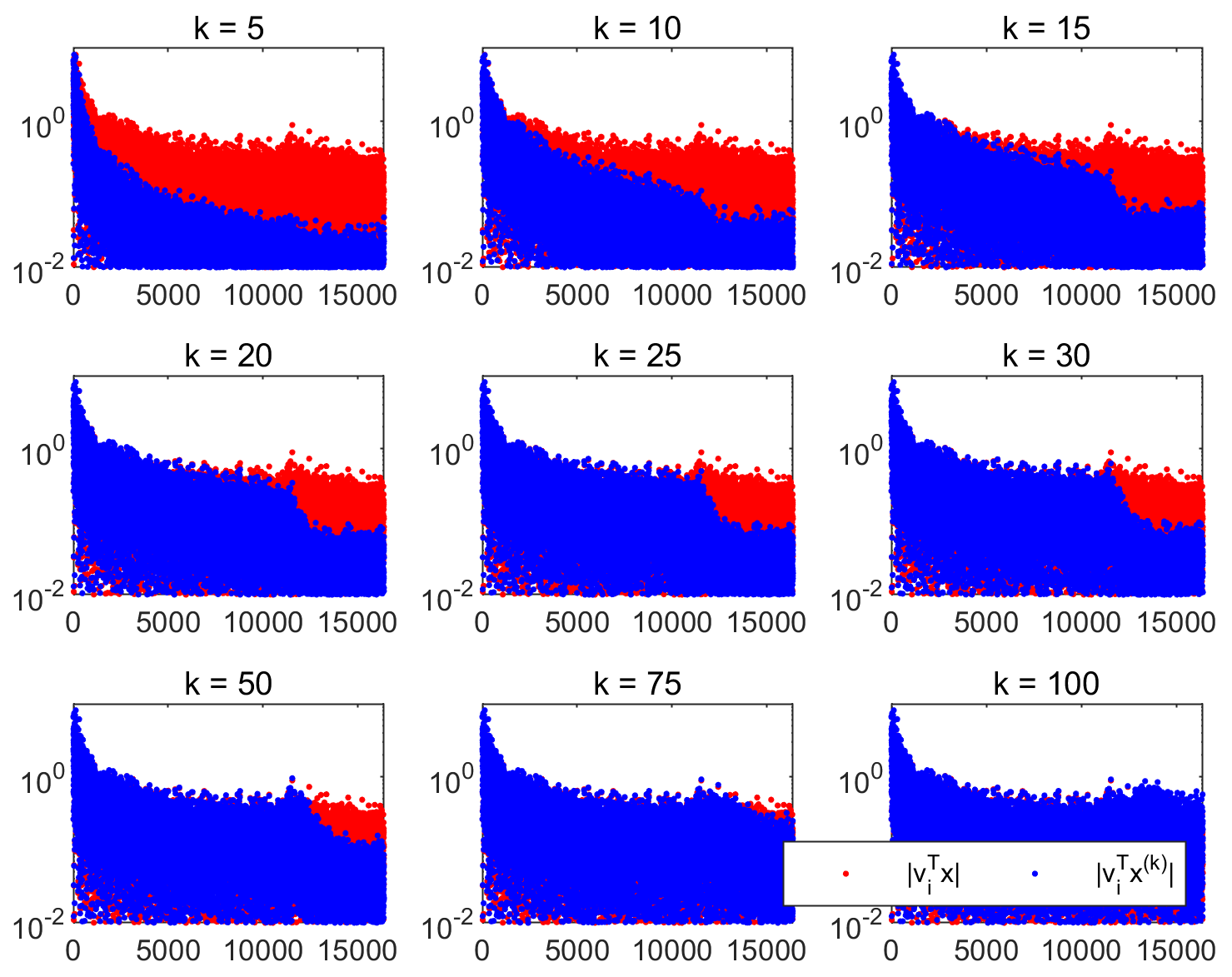}
	\includegraphics[width=0.35\textwidth]{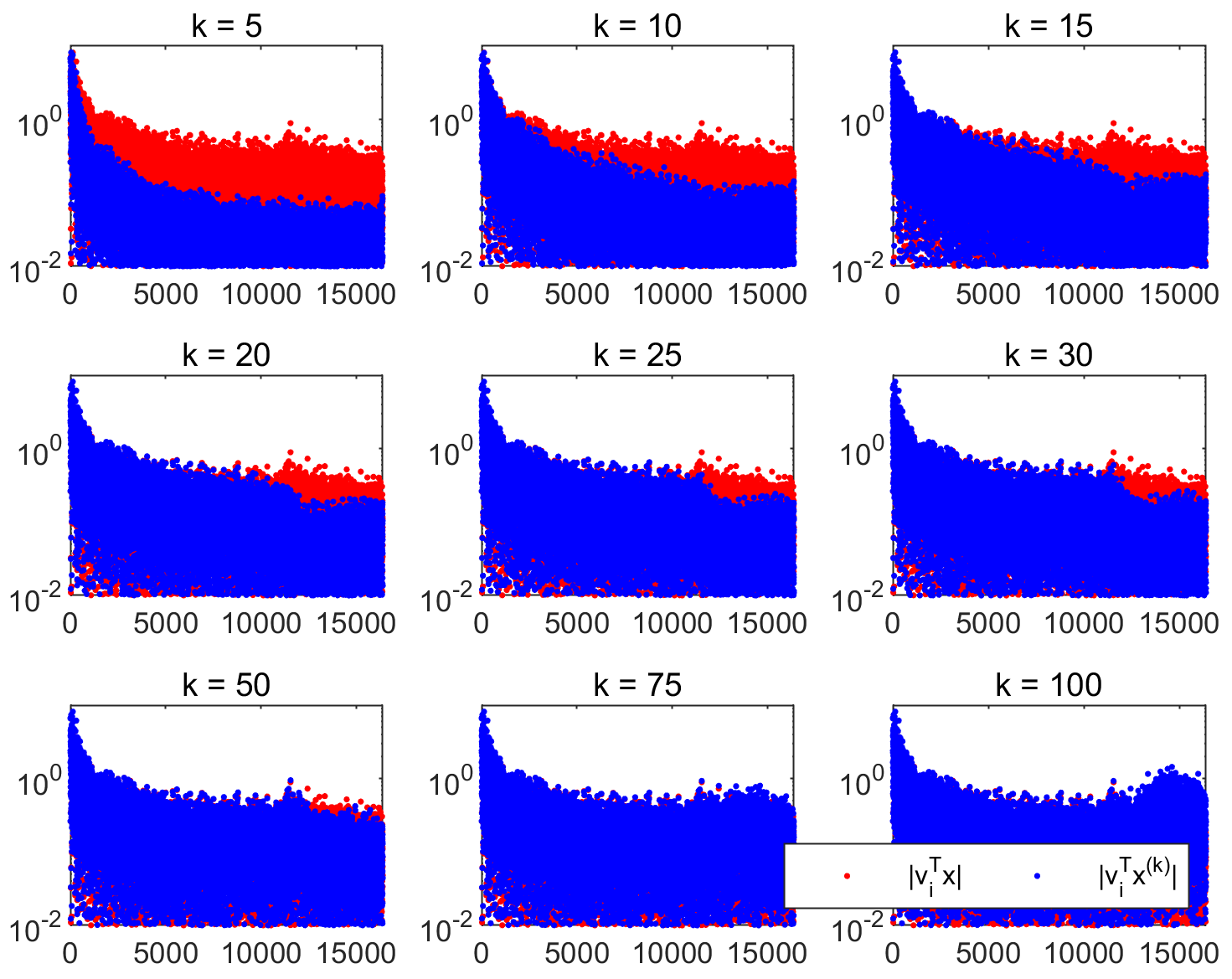}
	\caption{SVD analysis of the iterates of AB-GMRES (top) and BA-GMRES (bottom)
		for the unmatched transpose $B_\tau$ generated according to~\eqref{eq:Btau},
		similar to Figures \ref{fig:svd} and \ref{fig:svdagain}.
		The noise level is $\| e \|_2 / \| \bar{b} \|_2 = 0.003$.}
	\label{fig:SmallTallSVDerrrsInB}
\end{figure}

Similar to the SVD analysis in the previous subsection, we can use the SVD of
$\As$ to analyze the semi-convergence of AB- and BA-GMRES applied to $(\As,B_\tau)$
for $\tau=0.01$, $0.1$, $0.3$ and $0.5$.
For this analysis we use the small test matrix~$\As$ (because we need to compute
the full SVD) and Table~\ref{tbl:unmatch} shows
the corresponding $\| B_\tau - \AsT \|\fro / \| \As \|\fro$.
The relative noise level is $\| e \|_2 / \| \bar{b} \|_2 = 0.003$.
The results are shown in Figure \ref{fig:SmallTallSVDerrrsInB}.
\begin{enumerate}
	\item
	This behavior and the plots are practically identical for AB-GMRES and BA-GMRES.
	\item
	For $\tau=0.01$ the behavior is, as expected, almost similar to that for
	LSQR/LSMR shown in Figure~\ref{fig:svd} (which corresponds to the case $\tau=0$).
	Specifically, during the first iterations we do not include undesired noisy
	SVD components corresponding to the small singular values.
	\item
	As $\tau$ increases the behavior starts to resemble the case $(\As,\AlT)$
	shown in Figure~\ref{fig:svd}, where we -- even during the first iterations --
	include more undesired SVD components as compared to LSQR/LSMR.
\end{enumerate}
Again these results confirm our theory about the connection between
the unmatchedness $\| \AiT - \AsT \|\fro / \| \As \|\fro$ and the behavior
of the iterative methods.

\section{Conclusion}

The AB- and BA-GMRES algorithms can be considered as preconditioned
versions of the well-known GMRES algorithm, and we discuss how to use
these algorithms for solving large-scale X-ray CT problems
with an unmatched projector.
We also study the behavior of AB- and BA-GMRES as the back
projector approaches a matched one.
Numerical experiments, including numerical SVD analysis, provide the
insight that AB- and BA-GMRES exhibit semi-convergence and thus they behave
as regularizing iterative methods where the number of iterations
is the regularization parameter (similar to the behavior as LSQR/LSMR)\@.
Our numerical experiments also demonstrate that the discrepancy principle
and the NCP-criterion work well as stopping rules.
We recommend software developers to consider AB- and BA-GMRES
for the solution of large-scale CT reconstruction problems.

\section*{Acknowledgements}

We would like to thank Dr.\ Silvia Gazzola and Prof.\ Lothar Reichel
for valuable information.

\appendix
\section{Proof of Theorem~\ref{th:perturb_bound}}
\label{app:proof4perturb_bound}

\setcounter{lemma}{0}
\renewcommand{\thelemma}{\Alph{section}\arabic{lemma}}

We introduce notation required for the assertions below.
Denote the rank of $A$ by $r$.
Let $U = [U_1, U_2] \in \mathbb{R}^{m \times m}$ be an orthogonal matrix with $\mathcal{R}(U_1) = \mathcal{R}(A)$ and $\mathcal{R}(U_2) = \mathcal{R}(A)^\perp$, where $U_1 \in \mathbb{R}^{m \times r}$ and $U_2 \in \mathbb{R}^{m \times (m-r)}$.
Let $V = [V_1, V_2] \in \mathbb{R}^{n \times n}$ be an orthogonal matrix with $\mathcal{R}(V_1) = \mathcal{N}(A)^\perp$ and $\mathcal{R}(V_2) = \mathcal{N}(A)$, where $V_1 \in \mathbb{R}^{n \times r}$ and $V_2 \in \mathbb{R}^{n \times (n-r)}$.
Then, we transform the unmatched normal equations via $U$ and $V$ into
\begin{align}
	V^\trans \tilde{B} \tilde{A} U U^\trans (x_\mathrm{min} + \delta x_\mathrm{min}) = V^\trans \tilde{B} U U ^\trans b,
\end{align}
in which
\begin{align}
	U^\trans A V =
	\begin{bmatrix}
		A_{11} & \mathrm{O} \\
		\mathrm{O} & \mathrm{O}
	\end{bmatrix}, \quad
	U^\trans E_k V =
	\begin{bmatrix}
		E_{11}^{(k)} & E_{12}^{(k)} \\
		E_{21}^{(k)} & E_{22}^{(k)}
	\end{bmatrix}
\end{align}
for $k = 1$, $2$, where $A_{11} = U_1^\trans A V_1$, $E_{ij}^{(k)} = {U_i}^\trans E_k V_j$ for $i, j = 1, 2$.
Hereafter, we neglect higher-order terms to derive a first-order perturbation bound for $\delta x_\mathrm{min}$.
Thus, we obtain a first-order approximation of the transformed coefficient matrix of \eqref{eq:trans_unmatched_ne}
\begin{align}
	V^\trans \tilde{B}^\trans U U^\trans \tilde{A} V =
	\begin{bmatrix}
		G & {A_{11}}^\trans E_{12}^{(1)} \\
		(E_{12}^{(2)})^\trans A_{11} & \mathrm{O}
	\end{bmatrix} ,
	\label{eq:VTAhatAUUTAtildeV}
\end{align}
where $G = {A_{11}}^\trans A_{11} + F$ with $F = {A_{11}}^\trans E_{11}^{(1)} + (E_{11}^{(2)})^\trans A_{11}$, and the right-hand side vector of \eqref{eq:trans_unmatched_ne} together with $V^\trans$
\begin{align}
	V^\trans \tilde{B} U U^\trans b =
	\begin{bmatrix}
		{A_{11}}^\trans \bar{b}_1 + {A_{11}}^\trans \delta b_1 + (E_{11}^{(2)})^\trans \bar{b}_1 + (E_{21}^{(2)})^\trans \bar{b}_2 \\
		(E_{12}^{(2)})^\trans \bar{b}_1 + (E_{22}^{(2)})^\trans \bar{b}_2
	\end{bmatrix},
\end{align}
where $b_i = U_i^\trans b$, $i = 1, 2$, and $\delta b_1 = U_1^\trans \delta b$.

\begin{lemma} \label{lm:err_norm}
	Assume that $\tilde{A}$ and $\tilde{B}$ are both acute perturbations of $A$ and $A^\mathsf{T}$, respectively.
	The first-order bound of the relative error norm is given by
	\begin{align}
		\frac{\| \delta x_\mathrm{min} \|_2}{\| x_\mathrm{min} \|_2} \leq \kappa_2 (A) \left\lbrace \sigma_r^{-1} \left[ \left( \| E_{11}^{(1)} \|_2 + \| E_{12}^{(1)} \|_2 \right) \frac{\| \bar{b}_1 \|_2}{\| \bar{b} \|_2} + \frac{\|  (E_{21}^{(2)})^\trans \bar{b}_2 \|_2}{\| \bar{b} \|_2}  \right] + \frac{\| \delta b_1 \|_2}{\| \bar{b} \|_2} \right\rbrace,
	\end{align}
	where $\kappa_2(A) = \| A \|_2 \| A^\dag \|_2$ is the condition number of $A$ and $\sigma_r$ is the nonzero smallest singular value of $A$.
\end{lemma}

\begin{proof}[Proof of Lemma~\ref{lm:err_norm}]
	The assumptions imply that there is no vector in $\mathcal{R}(A)$ that is orthogonal to $\mathcal{R}(\tilde{A})$ and $\mathcal{R}(\tilde{B})$, $\mathrm{rank} (A) = \mathrm{rank}(\tilde{A}) = \mathrm{rank}(\tilde{B})$ \cite[Theorem~3.1]{StewartSun1990}, and $A_{11}$ is nonsingular \cite[Theorem~3.3]{StewartSun1990}.
	We use the formula in \cite{HungMarkham1975LAA} for the pseudoinverse of a block two-by-two matrix to \eqref{eq:VTAhatAUUTAtildeV}.
	Let $C = V^\trans \tilde{B}^\trans U U^\trans \tilde{A} V$.
	Then,
	\begin{align}
		C^\dag =
		\begin{bmatrix}
			(G^\dag)^\trans & (G^\trans G)^\dag A_{11}^\trans E_{12}^{(2)} \\
			(E_{12}^{(1)})^\trans A_{11} (G^\dag)^\trans G^\dag & \mathrm{O}
		\end{bmatrix}.
	\end{align}
	Therefore, the minimum-norm solution together with $V$ is given by
	\begin{align}
		V^\trans (x_\mathrm{min} + \delta x_\mathrm{min})
		& = C^\dag V^\trans \tilde{B} U U^\trans b \\
		& =
		\begin{bmatrix}
			(G^\dag)^\trans \left[ {A_{11}}^\trans (\bar{b}_1 + \delta b_1) + (E_{11}^{(2)})^\trans \bar{b}_1 + (E_{21}^{(2)})^\trans \bar{b}_2\right] \\
			(E_{12}^{(1)})^\trans A_{11} (G^\dag)^\trans G^\dag {A_{11}}^\trans \bar{b}_1
		\end{bmatrix}.
	\end{align}
	Applying the formula in \cite{HungMarkham1977} for the pseudoinverse of the sum of two matrices to $G^\dag = ( {A_{11}}^\trans A_{11} + F )^\dag$, we obtain $G^\dag = ({A_{11}}^\trans A_{11})^{-1} - L$, where $L = ({A_{11}}^\trans A_{11})^{-1} F ({A_{11}}^\trans A_{11})^{-1}$.
	Therefore, we have
	\begin{align}
		V^\trans (x_\mathrm{min} + \delta x_\mathrm{min})
		=
		\begin{bmatrix}				
			{A_{11}}^{-1} (\bar{b}_1 + \delta b_1) + ({A_{11}}^\trans A_{11})^{-1} \left[ (E_{11}^{(2)})^\trans \bar{b}_1 + (E_{21}^{(2)})^\trans \bar{b}_2\right] - L^\trans {A_{11}}^\trans \bar{b}_1 \\
			(E_{12}^{(1)})^\trans ({A_{11}}^\trans A_{11})^{-1} \bar{b}_1
		\end{bmatrix}.
	\end{align}
	As
	\begin{align}
		V^\trans x_\mathrm{min} =
		\begin{bmatrix}
			A_{11}^{-1} \bar{b}_1 \\
			\boldsymbol{0}
		\end{bmatrix},
	\end{align}
	we have
	\begin{align}
		V^\trans \delta x_\mathrm{min}
		=
		\begin{bmatrix}
			- {A_{11}}^{-1} E_{11}^{(1)} A_{11}^{-1} \bar{b}_1 + A_{11}^{-1} \delta b_1 + ({A_{11}}^\trans A_{11})^{-1} (E_{21}^{(2)})^\trans \bar{b}_2 \\
			(E_{12}^{(1)})^\trans (A_{11}^\trans A_{11})^{-1} \bar{b}_1
		\end{bmatrix}.
	\end{align}
	The proof is completed by bounding $\| V^\trans \delta x_\mathrm{min} \|_2 = \| \delta x_\mathrm{min} \|_2$, together with \break $\| x_\mathrm{min} \|_2 \geq \| A \|_2^{-1} \| \bar{b} \|_2$.
\end{proof}

\begin{proof}[Proof of Theorem~\ref{th:perturb_bound}]
	It follows from Lemma~\ref{lm:err_norm} that we have
	\begin{align}
		\| \delta x_\mathrm{min} \|_2
		& \leq \sigma_r^{-2} \left[ \left( \| P_{\mathcal{R}(A)} E_1 P_{\mathcal{R}(A^\trans)} \|_2 + \| P_{\mathcal{R}(A)} E_1 P_{\mathcal{N}(A)} \|_2 \right) \| P_{\mathcal{R}(A)} \bar{b} \|_2 + \| P_{\mathcal{R}(A^\trans)} E_2 P_{\mathcal{R}(A)^\perp} \bar{b} \|_2 \right] \\
		& \quad + \sigma_r^{-1} \| P_{\mathcal{R}(A)} \delta b \|_2 \\
		& \leq \sigma_r^{-2} \left( 2 \| E_1 \|_2 \| \bar{b}|_{\mathcal{R}(A)} \|_2 + \| E_2 \|_2 \| \bar{b}|_{\mathcal{R}(A)^\perp} \|_2 \right) + \sigma_r^{-1} \| \delta b|_{\mathcal{R}(A)} \|_2.
	\end{align}
\end{proof}

An alternative proof of Theorem~\ref{th:perturb_bound} can be given by directly applying the formula in \cite{HungMarkham1977} to the sum of $A^\trans A$ and $A^\trans E_1 + E_2^\mathsf{T} A$.

\begin{proof}[Alternative proof of Theorem~\ref{th:perturb_bound}]
	We use the formula in \cite{HungMarkham1977} for the pseudoinverse of the sum of two matrices~$(A^\trans A+E)^\dag$, where $E = A^\trans E_1 + E_2^\mathsf{T} A$.
	For convenience, let $M = A^\trans A E + E^\trans A^\trans A$ and
	\begin{align}
		T_n = (t_{ij}) \in \mathbb{R}^{n \times n}, \quad t_{ij} =
		\begin{cases}
			1, \quad j = n - i + 1, \\
			0, \quad \text{otherwise}.
		\end{cases}
	\end{align}
	Note ${T_n}^2 = \mathrm{I}$.
	Then,
	\begin{align}
		& (A^\trans A + E)^\dag \\
		& = T_n \left( T_n \left[ (A^\trans A)^2 \right]^\dag T_n \left\{ T_n - T_n M \left[ (A^\trans A)^2 \right]^\dag \right\} + T_n P_{\mathcal{N}(A)} M \left[ (A^\trans A)^4 \right]^\dag \right) (A^\trans A + E^\trans) \\
		& = \left\{ \mathrm{I} - \left[ (A^\trans A)^2 \right]^\dag M + P_{\mathcal{N}(A)} M \left[ (A^\trans A)^2 \right]^\dag \right\} \left[ (A^\trans A)^2 \right]^\dag (A^\trans A + E^\trans) \\
		& = \left\{ \mathrm{I} - \left[ (A^\trans A)^2 \right]^\dag M + P_{\mathcal{N}(A)} M \left[ (A^\trans A)^2 \right]^\dag \right\} (A^\trans A)^\dag + \left[ (A^\trans A)^2 \right]^\dag E^\trans \\
		& = (A^\trans A)^\dag - \left[ (A^\trans A)^2 \right]^\dag (A^\trans A E + E^\trans A^\trans A) (A^\trans A)^\dag + P_{\mathcal{N}(A)} (A^\trans A E + E^\trans A^\trans A) [ (A^\trans A)^3 ]^\dag \\
		& \quad + \left[ (A^\trans A)^2 \right]^\dag E^\trans \\
		& = (A^\trans A)^\dag -  (A^\trans A)^\dag E (A^\trans A)^\dag - \left[ (A^\trans A)^2 \right]^\dag E^\trans P_{\mathcal{R}(A^\trans)} + P_{\mathcal{N}(A)} E^\trans \left[ (A^\trans A)^2 \right]^\dag \\
		& \quad + \left[ (A^\trans A)^2 \right]^\dag {E_1}^\trans A + (A A^\trans A)^\dag {E_2}^\trans \\
		& = (A^\trans A)^\dag -  A^\dag E_1 (A^\trans A)^\dag - (A^\trans A)^\dag E_2 (A^\dag)^\trans - \left[ (A^\trans A)^2 \right]^\dag {E_1}^\trans A - (A A^\trans A)^\dag {E_2}^\trans P_{\mathcal{R}(A^\trans)} \\
		& \qquad \qquad + P_{\mathcal{N}(A)} {E_1}^\trans (A^\trans A A^\trans)^\dag + \left[ (A^\trans A)^2 \right]^\dag {E_1}^\trans A + (A A^\trans A)^\dag {E_2}^\trans \\
		& = (A^\trans A)^\dag -  A^\dag E_1 (A^\trans A)^\dag - (A^\trans A)^\dag E_2 (A^\dag)^\trans - (A A^\trans A)^\dag {E_2}^\trans P_{\mathcal{R}(A^\trans)} + P_{\mathcal{N}(A)} {E_1}^\trans (A^\trans A A^\trans)^\dag \\
		& \quad + (A A^\trans A)^\dag {E_2}^\trans.
	\end{align}
	Applying this matrix to $\tilde{B}$, we have
	\begin{align}
		(A^\trans A + E)^\dag \tilde{B} = A^\dag - A^\dag E_1 A^\dag - (A^\trans A)^\dag E_2^\trans P_{\mathcal{R}(A)} + P_{\mathcal{N}(A)} {E_1}^\trans (A A^\trans)^\dag + (A^\trans A)^\dag E_2^\trans.
	\end{align}
	Hence, the minimum-norm solution of \eqref{eq:trans_unmatched_ne} is given by
	\begin{align}
		& (\tilde{B} \tilde{A})^\dag \tilde{B} b \\
		& = (A^\trans A + E)^\dag (A^\trans + E_2^\trans) (\bar{b} + \delta b) \\
		& = \left[ A^\dag - A^\dag E_1 A^\dag - (A^\trans A)^\dag E_2^\trans P_{\mathcal{R}(A)} + P_{\mathcal{N}(A)} {E_1}^\trans (A A^\trans)^\dag + (A^\trans A)^\dag E_2^\trans \right] (\bar{b}|_{\mathcal{R}(A)} + \bar{b}|_{\mathcal{R}(A)^\perp}) \\
		& \quad + A^\dag (\delta b|_{\mathcal{R}(A)} + \delta b|_{\mathcal{R}(A)^\perp}) \\
		& = \left[ A^\dag - A^\dag E_1 A^\dag + P_{\mathcal{N}(A)} {E_1}^\trans (A A^\trans)^\dag \right] \bar{b}|_{\mathcal{R}(A)} + (A^\trans A)^\dag E_2^\trans \bar{b}|_{\mathcal{R}(A)^\perp} + A^\dag \delta b|_{\mathcal{R}(A)}.
	\end{align}
	Therefore, the error is given by
	\begin{align}
		\delta x_\mathrm{min} & = (\tilde{B} \tilde{A})^\dag \tilde{B} b - A^\dag \bar{b} \\
		& = \left[ - A^\dag E_1 A^\dag + P_{\mathcal{N}(A)} {E_1}^\trans (A A^\trans)^\dag \right] \bar{b}|_{\mathcal{R}(A)} + (A^\trans A)^\dag E_2 \bar{b}|_{\mathcal{R}(A)^\perp} + A^\dag \delta b|_{\mathcal{R}(A)}.
	\end{align}
	The proof can be completed by bounding $\| \delta x_\mathrm{min} \|_2$.
\end{proof}

\bibliographystyle{siam}
\bibliography{mybibfile}
\end{document}